\documentclass{article}
\usepackage{amsthm,amssymb, amsmath, amscd}
\usepackage{stmaryrd}
\usepackage[usenames, dvipsnames]{color}
\usepackage[pdfborder={0 0 0}, colorlinks = true, linkcolor=blue, citecolor=OliveGreen]{hyperref}
\usepackage[nameinlink]{cleveref}
\usepackage{tikz-cd}
\usepackage{enumerate}
\usepackage{mathrsfs}
\usepackage{bm}
\usepackage{todonotes}

\usepackage{geometry}

\theoremstyle{plain}
\newtheorem{theorem}{Theorem}[section]
\newtheorem{lemma}[theorem]{Lemma}
\newtheorem{corollary}[theorem]{Corollary}
\newtheorem{proposition}[theorem]{Proposition}

\theoremstyle{remark}
\newtheorem{remark}[theorem]{Remark}

\theoremstyle{definition}
\newtheorem{definition}[theorem]{Definition}

\newcommand{\bbA}{\mathbb A}

\newcommand{\bbC}{\mathbb C}
\newcommand{\bbD}{\mathbb D}

\newcommand{\bbH}{\mathbb H}
\newcommand{\bbP}{\mathbb P}
\newcommand{\bbQ}{\mathbb Q}
\newcommand{\bbR}{\mathbb R}

\newcommand{\calC}{\mathcal C}
\newcommand{\calD}{\mathcal D}
\newcommand{\calE}{\mathcal E}

\newcommand{\calO}{\mathcal O}

\newcommand{\calX}{\mathcal X}

\newcommand{\calZ}{\mathcal Z}

\newcommand{\bfu}{\mathbf u}
\newcommand{\bfv}{\mathbf v}
\newcommand{\bfx}{\mathbf x}
\newcommand{\bfy}{\mathbf y}

\newcommand{\bftau}{\bm{\tau}}
\newcommand{\bflambda}{\bm{\lambda}}

\newcommand{\lie}[1]{ \mathfrak{#1}}
\newcommand{\psm}[1]{\left( \begin{smallmatrix}#1 \end{smallmatrix} \right)}

\newcommand{\ChowHat}[1]{\widehat{\mathrm{CH}}{}^{#1}}

\newcommand{\ddc} {\mathrm{dd^c}}
\newcommand{\rk}{ \mathrm{rk}}
\newcommand{\rank}{\mathrm{rank}}
\newcommand{\Id}{\mathrm{Id}}

\newcommand{\Dcur}{\calD_{\mathrm{cur}}}
\newcommand{\FJhat}{\widehat{\mathrm{FJ}}}
\newcommand{\can}{\mathrm{can}}

\DeclareMathOperator{\Sym}{Sym}
\DeclareMathOperator{\Sp}{Sp}
\DeclareMathOperator{\Sptilde}{\widetilde{Sp}}

\newcommand{\GL}{\mathrm{GL}}
\newcommand{\Mat}{\mathrm{Mat}}
\newcommand{\SL}{\mathrm{SL}}

\newcommand{\KM}{\mathrm{KM}}
\newcommand{\ev}{\mathrm{ev}}

\numberwithin{equation}{section}

\title{Arithmetic special cycles and Jacobi forms}
\date{}
\author{Siddarth Sankaran}
\begin{document}
	\maketitle
	
	\begin{abstract}
		We consider families of  special cycles, as introduced by Kudla, on  Shimura varieties attached to anisotropic quadratic spaces over totally real fields.  By augmenting these cycles with Green currents, we obtain classes in the arithmetic Chow groups of the canonical models of these Shimura varieties (viewed as arithmetic varieties over their reflex fields). The main result of this paper asserts that  generating series built from these cycles can be identified with the Fourier expansions of non-holomorphic Hilbert-Jacobi modular forms. This result provides evidence for an arithmetic analogue of Kudla's conjecture relating these cycles to Siegel modular forms. 
	\end{abstract}
	\tableofcontents
	
	\section{Introduction}
	
	The main result of this paper is a modularity result for certain generating series of ``special" cycles that live in the arithmetic Chow groups of Shimura varieties of orthogonal type. 
	
	We begin by introducing the main players. Let $F$ be a totally real extension of $\mathbb Q$  with $d = [F:\bbQ]$, and let $\sigma_1, \dots , \sigma_d$ denote the archimedean places of $F$. 
	Suppose $V$ is a quadratic space over $F$ that is of signature $((p,2), (p+2,0), (p+2,0), \cdots , (p+2,0))$  with $p>0$. In other words, we assume that $V \otimes_{F, \sigma_1} \bbR$ is a real quadratic space of signature $(p,2)$ and that $V$ is positive definite at all other real places. 
	
	\emph{We assume throughout that $V$ is anisotropic over F. } Note that  the signature condition guarantees that $V$ is anisotropic whenever $d > 1$.
	
	Let $H = \mathrm{Res}_{F/\bbQ} \mathrm{GSpin}(V)$. The corresponding Hermitian symmetric space $\bbD$ has two connected components; fix one component $\bbD^+$ and let $H^+(\bbR)$ denote its stabilizer in $H(\bbR)$. For a neat compact open subgroup $K_f \subset H(\bbA_f)$, let $\Gamma := H^+(\bbR) \cap H(\bbQ) \cap K_f$, and consider the quotient 
	\begin{equation}
	X(\bbC) \ := \ \Gamma \big\backslash \bbD^+.
	\end{equation}
	This space is a (connected) Shimura variety; in particular, it admits a canonical model $X$ over a  number field  $E \subset \bbC$ depending on $K_f$, see \cite{KudlaAlgCycles} for details. Moreover, as $V$ is anisotropic,   $X$ is a projective variety.
	
	Fix a $\Gamma$-invariant lattice $L \subset V$ such that the restriction of the bilinear form $\langle \cdot, \cdot, \rangle$ to $L$ is valued in $\calO_F$, and consider the dual lattice
	\begin{equation}
		L' = \{  \bfx \in L \ | \ \langle \bfx, L \rangle \subset \partial_F^{-1} \}
	\end{equation}
	where $\partial_F^{-1}$ is the inverse different.

 For an integer $n$ with $1 \leq n \leq p$,  let $S(V(\bbA_f)^n)$ denote the Schwartz space of compactly supported, locally constant functions on $V(\bbA_f)^n$, and consider the subspace
	\begin{equation} \label{eqn:intro S(L) defn}
			S(L^n) \ := \ \{  \varphi \in S(V(\bbA_f)^n)^{\Gamma} \ | \ \mathrm{supp}(\varphi) \subset (\widehat{L'})^n \text{ and }  \varphi(\bfx + l) = \varphi(\bfx) \text{ for  all } l \in L^n  \}.
	\end{equation}
	
	For every $T \in \Sym_n(F)$ and $\Gamma$-invariant Schwartz function $\varphi \in S(L^n)$, there is an $E$-rational  ``special" cycle
	\begin{equation}
	Z(T,\varphi)
	\end{equation}
	of codimension $n$ on $X$, defined originally by Kudla \cite{KudlaAlgCycles}; this construction is reviewed in  \Cref{sec:Prelim cycles} below. 
	
	It was conjectured by Kudla that these cycles are closely connected to automorphic forms; more precisely, he conjectured that upon passing to the Chow group of $X$, the generating series formed by the classes of these special cycles can be identified with the Fourier expansions of Hilbert-Siegel modular forms. When $F = \mathbb Q$,  the codimension one case of this conjecture follows from results of Borcherds \cite{BorcherdsGKZ}, and the conjecture for higher codimension was established by Zhang and Bruinier-Raum \cite{BruinierRaum,ZhangThesis};  when $F \neq \bbQ$, conditional proofs have been given by Yuan-Zhang-Zhang \cite{YuanZhangZhang} and Kudla \cite{KudlaTotallyReal}.
	
	More recently, attention has shifted to the  arithmetic analogues of this result, where one replaces the Chow groups with an ``arithmetic" counterpart, attached to a model $\calX$ of $X$ defined over a subring of the reflex field of $E$;  these arithmetic Chow groups were introduced by Gillet-Soul\'e \cite{GilletSouleIHES} and subsequently generalized by Burgos-Kramer-K\"uhn \cite{BurgosKramerKuhn}.  Roughly speaking, in this framework cycles are represented by pairs $(\calZ, g_{\calZ})$, where $\calZ$ is a cycle on $\calX$, and $g_{\calZ}$ is a \emph{Green object}, a purely differential-geometric datum that encodes  cohomological information about the archimedean fibres of $\calZ$. 
	
	In this paper, we consider the case where the model $\calX$ is taken to be $X$ itself. In order to promote the special cycles to the arithmetic setting, we need to choose the Green objects: for this, we employ the results of \cite{GarciaSankaran}, where a family $\{ \lie g(T,\varphi; \bfv)\}$ of Green forms was constructed. Note that these forms depend on an additional parameter $\bfv \in \Sym_{n}(F \otimes_{\bbQ} \bbR)_{\gg 0}$, which should be regarded as the imaginary part of a variable in the Hilbert-Siegel upper half space. 

	We thereby obtain classes 
	\begin{equation}
		\widehat Z(T,\bfv) \in \ChowHat{n}_{\bbC}(X) \otimes_{\bbC} S(L^n)^{\vee} ,
	\end{equation}
	where $\ChowHat{n}_{\bbC}(X) $ is the Gillet-Soul\'e arithmetic Chow group attached to $X$, by the formula
	\begin{equation}
	\widehat Z(T,\bfv)(\varphi) \ =\  \big(Z(T,\varphi), \, \lie g(T, \varphi; \bfv) \big)  \ \in \ \ChowHat{n}_{\bbC}(X).
	\end{equation}

	For reasons that will emerge in the course of the proof our main theorem, we will also need to consider a larger arithmetic Chow group $\ChowHat{n}_{\bbC}(X, \Dcur)$, constructed by  Burgos-Kramer-K\"uhn \cite{BurgosKramerKuhn}. This group appears as an example of their  general cohomological approach to the theory of Gillet-Soul\'e. There is a natural injective  map $\ChowHat{n}_{\bbC}(X) \hookrightarrow \ChowHat{n}_{\bbC}(X,\Dcur)$; abusing notation, we identify the special cycle $\widehat Z(T,\bfv)$ with its image under this map.
	\begin{theorem} \label{thm:Intro main theorem}
		\begin{enumerate}[(i)]
		\item 	Suppose $1 < n \leq p$. Fix $T_2 \in \Sym_{n-1}(F)$, and define the formal generating series
			\begin{equation} \label{eqn:main thm n}
				\FJhat_{T_2}(\bftau) \ = \ \sum_{T = \psm{ * & * \\ * & T_2}} \widehat Z(T,\bfv) \, q^T
			\end{equation}
			where $\bftau \in \bbH_n^d$ lies in the Hilbert-Siegel upper half space  of genus $n$, and $\bfv = \mathrm{Im}(\tau)$. Then $\FJhat_{T_2}(\bftau)$ is the $q$-expansion of a (non-holomorphic) Hilbert-Jacobi modular form of weight $p/2 + 1$ and index $T_2$, taking values in $\ChowHat{n}_{\bbC}(X,\Dcur) \otimes S(L^n)^{\vee}$ via the Weil representation. 
			\item When $n=1$, the generating series
				\begin{equation} \label{eqn:main thm genus 1}
					\widehat \phi_1(\bftau)  \ = \  \sum_{t \in F} \widehat Z(t,\bfv) q^t
				\end{equation}
				is the $q$-expansion of a (non-holomorphic) Hilbert modular form of weight $p/2+1$, valued in $\ChowHat{1}_{\bbC}(X) \otimes S(L)^{\vee}$. 
		\end{enumerate}
	\end{theorem}

Some clarification is warranted in the interpretation of this theorem. The issue is that there  is no apparent topology on the arithmetic Chow groups for which the series \eqref{eqn:main thm n} and \eqref{eqn:main thm genus 1} can be said to converge in a reasonable sense; in a similar vein,  the Green forms $\lie g(T,\bfv)$ vary smoothly in the parameter $\bfv$, but there does not appear to be a natural way  in which the family of classes $\widehat Z(T,\bfv)$ can be said to vary smoothly. 
What is being asserted in the  theorem is the existence of: 
\begin{enumerate}[(i)]
	\item 	finitely many classes $\widehat Z_1, \dots \widehat Z_r \in \widehat{\mathrm{CH}}{}^n(X, \mathcal D_{\mathrm{cur}})$ (or in $\ChowHat{1}_{\bbC}(X)$ when $n=1$),
	\item finitely many Jacobi modular forms (in the usual sense) $f_1, \dots, f_r$,
	\item  and a  Jacobi form $g(\bftau)$ valued in the space of currents on $X$ that is locally uniformly bounded in $\bftau$,
\end{enumerate} 
such that	the $T$'th coefficient of the Jacobi form $
\sum_i f_i(\bftau) \widehat Z_i + a( g(\bftau)) $
coincides with $\widehat Z(T,\bfv)$. Here $a(g(\bftau)) \in \widehat{\mathrm{CH}}{}^n(X, \mathcal D_{\mathrm{cur}})$ is an ``archimedean class" associated to the current $g(\bftau)$. A more precise account may be found in \Cref{sec:modularity def}. 

To prove the theorem, we first prove the $n=1$ case, using a modularity result due to Bruinier \cite{Bruinier-totally-real} that involves a different set of Green functions; the theorem in this case follows from a comparison between his Green functions and ours. 

For $n>1$, we exhibit a decomposition
\begin{equation} \label{eqn:intro Zhat decomp}
\widehat Z(T,\bfv) = \widehat A(T,\bfv) + \widehat B(T,\bfv)
\end{equation}
in $\ChowHat{n}(X,\Dcur) \otimes S(L^n)^{\vee}$; this decomposition is based on a mild generalization of the star product formula \cite[Theorem 4.10]{GarciaSankaran}. The main theorem then  follows from the modularity of the series
\begin{equation}
\widehat \phi_A(\bftau) :=  \sum_{T = \psm{ * & * \\ * & T_2}} \widehat A(T,\bfv) \, q^T \qquad \text{and} \qquad \widehat \phi_B(\bftau) \ = \ \sum_{T = \psm{ * & * \\ * & T_2}} \widehat B(T,\bfv) \, q^T ,
\end{equation}
which are proved in \Cref{thm:A series} and \Cref{thm:B series} respectively.  The classes $\widehat A(T,\bfv)$ are expressed as linear combinations of pushforwards of special cycles along sub-Shimura varieties of $X$, weighted by the Fourier coefficients of classical theta series; the modularity of $\widehat \phi_A(\bftau)$ follows from this description and the $n=1$ case. 
The classes $\widehat B(T,\bfv)$ are purely archimedean, and the modularity of $\widehat \phi_B(\bftau)$ follows from an explicit computation involving the Kudla-Millson Schwartz form, \cite{KudlaMillsonIHES}.

This result provides evidence for the arithmetic version of  Kudla's conjecture, namely that the generating series
\begin{equation}  
		\widehat \phi_n(\bftau) \ = \ \sum_{T \in \Sym_n(F)} \, \widehat Z(T,\bfv) \, q^T
\end{equation}
is a Hilbert-Siegel modular form; indeed, the series $\FJhat_{T_2}(\bftau)$ is a formal Fourier-Jacobi coefficient of $\widehat\phi_n(\bftau)$. Unfortunately, there does not seem to be an obvious path by which one can infer the more general result from the results in this paper; the decomposition $\widehat Z(T,\bfv)$ depends on the lower-right matrix $T_2$, and it is not clear how to compare the decompositions for various $T_2$.

	\subsection*{Acknowledgements}
	The impetus for this paper emerged from discussions during an AIM SQuaRE workshop; I'd like to thank the participants -- Jan Bruinier, Stephan Ehlen, Stephen Kudla and Tonghai Yang -- for the stimulating discussion and insightful remarks, and AIM for the hospitality. I'd also like to thank Craig Cowan for a helpful discussion on the theory of currents. This work was partially supported by an NSERC Discovery grant.

	
\section{Preliminaries} \label{sec:Prelims}
	
		\subsection{Notation}
			\begin{itemize}
				\item Throughout, we fix a totally real field $F$ with $[F:\bbQ]=d$. Let $\sigma_1 , \dots, \sigma_d$ denote the real embeddings. Using these embeddings, we identify $F\otimes_{\bbQ} \bbR$ with $\mathbb R^d$, and denote by $\sigma_i(\mathbf t)$ the $i$'th component of $\mathbf t \in F \otimes_{\bbQ}\bbR$ under this identification.
				\item For any matrix $A $, we denote the transpose by $A'$.
				\item If $A \in \Mat_{n}(F \otimes_{\bbQ}\bbR)$, we write
					\begin{equation}
						e(A) \ := \  \prod_{i=1}^d \, \exp\big(  2 \pi i\, \mathrm{tr} \left( \sigma_i(A)  \right)\big)
					\end{equation}
					\item If $(V,Q)$ is a quadratic space over $F$, let $\langle \bfx, \bfy \rangle$ denote the corresponding bilinear form. Here we take the convention $Q(\bfx) = \langle \bfx, \bfx \rangle$. If $\bfx \in V$ and $\bfy = (\bfy_1, \dots , \bfy_{n}) \in V^n$ , we set $\langle \bfx, \bfy \rangle  = ( \langle    \bfx, \bfy_1 \rangle , \dots, \langle \bfx, \bfy_{n} \rangle ) \in \Mat_{1,n}(F)$. 
					
					\item For $i = 1, \dots d$, we set $V_i = V \otimes_{F,\sigma_i} \bbR$. 
					
					\item 	Let 
					\begin{equation}
					\bbH_n^d = \{ \bftau = \bfu + i\bfv \in \Sym_n(F \otimes_{\bbQ} \bbR) \ | \ \bfv \gg 0 \}
					\end{equation} 
					denote the Hilbert-Siegel upper half-space of genus $n$ attached to $F$. Via the fixed embeddings $\sigma_1, \dots, \sigma_d$, we may identify $\Sym_n(F \otimes \bbR) \simeq \Sym_n(\bbR)^d$; we let $\sigma_i(\bftau) = \sigma_i(\bfu) + i \sigma_i(\bfv)$ denote the corresponding component, so that, in particular, $\sigma_i(\bfv) \in \Sym_n(\bbR)_{>0}$ for $i=1,\dots, d$. 
					
					If $\bftau \in \bbH_n^d$ and $T \in \Sym_n(F)$,  we write
						\begin{equation}
							q^T  = e (\bftau T). 
						\end{equation}
					
			\end{itemize}

		\subsection{Arithmetic Chow groups}
		In this section, we recall the theory of arithmetic Chow groups $\ChowHat{n}_{\bbC}(X)$, as conceived by Gillet and Soul\'e \cite{GilletSouleIHES}; note here and throughout this paper, we work with complex coefficients. Recall that $X$ is defined over a number field $E$ endowed with a fixed complex embedding $\sigma: E \to \bbC$. We view $X$ as an arithmetic variety over the ``arithmetic ring" $(E, \sigma, \text{complex conjugation})$ in the terminology of \cite[\S 3.1.1]{GilletSouleIHES}.

		An \emph{arithmetic cycle} is a pair $(Z,g)$, where $Z$ is a formal $\bbC$-linear combination of codimension $n$ subvarieties of $X$, and $g$ is a \emph{Green current} for $Z$; more precisely, $g$ is a current of degree $(p-1,p-1)$  on $X(\bbC)$ such Green's equation 
		\begin{equation} \label{eqn:ChowHat Green}
			\ddc g \ + \ \delta_{Z(\bbC)} = \omega
		\end{equation}
		holds, where the right hand side is the current defined by integration\footnote{Here and throughout this paper, we will abuse notation and write $\omega$ both for the form and the current it defines.} against some smooth form $\omega$. Given a codimension $n-1$ subvariety $Y$ and a rational function $f \in k(Y)^{\times}$ on $Y$, let
		\begin{equation}
		\widehat{\mathrm{div}} (f)  := (\mathrm{div}(f), - \log | f |^2 \, \delta_Y)
		\end{equation}
		denote the corresponding principal arithmetic divisor. The arithmetic Chow group $\ChowHat{n}_{\bbC}(X)$ is quotient of the space of arithmetic cycles by the subspace spanned by  (a) the principal arithmetic divisors and (b) classes of the form $(0, \eta)$ with $\eta \in \mathrm{im}(\partial) + \mathrm{im}(\overline \partial)$. For more details, see \cite{GilletSouleIHES,SouleBook}

		In their paper \cite{BurgosKramerKuhn}, Burgos, Kramer and K\"uhn give an abstract reformulation and generalization of this theory: their main results describe the construction of an arithmetic Chow group $\ChowHat{*}(X, \calC)$ attached to a ``Gillet complex" $\calC$. One of the examples they describe is the group attached to the complex of currents $\calD_{\mathrm{cur}}$; we will content ourselves with the superficial description of this group given below, which will suffice for our purposes, and the reader is invited to consult \cite[\S 6.2]{BurgosKramerKuhn} for a thorough treatment.
		
		Unwinding the formal definitions in \cite{BurgosKramerKuhn}, one finds that classes in $\ChowHat{n}_{\bbC}(X, \Dcur)$ are represented by tuples $(Z, [T, g])$, with $Z$ as before, but now  $T$ and $g$ are currents of degree $(n,n)$ and $(n-1, n-1)$ respectively such that\footnote{The reader is cautioned that in \cite{BurgosKramerKuhn}, the authors normalize delta currents and currents defined via integration by powers of $2 \pi i$, resulting in formulas that look slightly different from those presented here; because we are working with $\bbC$-coefficients, the formulations are equivalent.}
		\begin{equation}
			\ddc g + \delta_{Z(\bbC)}  = T + \ddc(\eta)
		\end{equation}
		for some current $\eta$ with support contained in $Z(\bbC)$;
		we can view this as a relaxation of the condition that the right hand side of \eqref{eqn:ChowHat Green} is smooth. A nice consequence of this description is that any codimension $n$ cycle $Z$ on $X$ gives rise to a \emph{canonical class} (see \cite[Definition 6.37]{BurgosKramerKuhn})
		\begin{equation}
			\widehat Z^{\can} \ := \ \left( Z, [\delta_Z, 0] \right).
		\end{equation}

		By \cite[Theorem 6.35]{BurgosKramerKuhn}, the natural map 
		\begin{equation}
			\ChowHat{n}_{\bbC}(X) \to \ChowHat{n}_{\bbC}(X,\Dcur), \qquad  (Z,g) \mapsto (Z, [\omega, g])
		\end{equation}   is injective. Moreover while $\ChowHat{*}_{\bbC}(X, \Dcur)$ is not a ring in general, it is a module over $\ChowHat{*}_{\bbC}(X)$.  As a special case of this product, let $(Z,g) \in \ChowHat{1}_{\bbC}(X)$ be an arithmetic divisor, where  $g$ is a Green function with logarithmic singularities along the divisor $Z$. Suppose $\widehat Y^{\can} \in \ChowHat{m}(X, \Dcur)$ is the canonical class attached to a cycle $Y$ that intersects $Z$ properly; then by inspecting the proofs of \cite[Theorem 6.23, Proposition 6.32]{BurgosKramerKuhn} we find
		\begin{equation}
			(Z,g) \cdot \widehat Y^{\can} \ = \  \left(Z \cdot Y, [\omega \wedge \delta_{Y(\bbC)}, g \wedge \delta_{Y(\bbC)}] \right)  \in \ChowHat{m+1}_{\bbC}(X,\Dcur).
		\end{equation}

		\begin{remark} \label{remark:Chow vanishing}
			One consequence of our setup is the vanishing of certain ``archimedean rational" classes in $\ChowHat{n}(X)$ and $\ChowHat{n}(X, \Dcur)$. More precisely, if  $Y$ is a codimension $n-1$ subvariety, then 
			\begin{equation}
			(0, \delta_{Y(\bbC)}) = 0  \ \in \ \ChowHat{n}_{\bbC}(X).
			\end{equation}
			To see this, let $c \in \bbQ $  be any rational number such that $c \neq 0$ or $1$, and view $c$ as a rational function on $Y$; its divisor is trivial, and so
			\begin{equation}
				0 = \widehat{\mathrm{div}}(c) \ = \ (0, - \log|c|^2 \delta_{Y(\bbC)}) \ = \ - \log|c|^2  \cdot (0, \delta_{Y(\bbC)}).
			\end{equation}
			and hence $(0, \delta_{Y(\bbC)}) = 0$.  As a special case, we have $(0,1) = 0 \in \ChowHat{1}_{\bbC}(X)$. 
		\end{remark}

		\subsection{Special cycles} \label{sec:Prelim cycles}
		Here we review Kudla's construction of the family $\{Z(T)\}$ of special cycles on $X$, \cite{KudlaAlgCycles}.
		First, recall that the symmetric space $\bbD$ has a concrete realization 
		\begin{equation} \label{eqn:symmetric space proj model}
			\bbD = \{ z \in \bbP^1(V\otimes_{\sigma_1, F} \bbC) \ | \ \langle z, z \rangle = 0, \langle z , \overline z \rangle < 0 \}
		\end{equation}
		where $\langle \cdot, \cdot\rangle$ is the $\bbC$-bilinear extension of the bilinear form on $V$; the two connected components of $\bbD$ are interchanged by conjugation.
		
		Given a collection of vectors $\bfx = (\bfx_1, \dots, \bfx_n) \in V^n$, let
		\begin{equation}
			\bbD^+_{\bfx} \ := \ \{  z \in \bbD^+  \ | \ z \perp \sigma_1(\bfx_i) \ \text{ for } i = 1, \dots, n\}.
		\end{equation} 
		where, abusing notation, we denote by $\sigma_1 \colon V \to V_1 = V \otimes_{F, \sigma_1} \bbR$ the map induced by inclusion in the first factor. 
		
		Let $\Gamma_{\bfx}$ denote the pointwise stabilizer of $\bfx$ in $\Gamma$; then the inclusion $\bbD^+_{\bfx} \subset \bbD^+$ induces a map
		\begin{equation}
		 \Gamma_{\bfx} \big\backslash \bbD^+_{\bfx} \to \Gamma \big\backslash \bbD^+ = X,
		\end{equation}
		which defines a complex algebraic cycle that we denote $Z(\bfx)$. 
		If the span of $\{ \bfx_1, \dots, \bfx_r \}$ is not totally positive definite, then $\bbD^+_{\bfx}= \emptyset$ and $Z(\bfx) = 0$; otherwise, the codimension of $Z(\bfx)$ is the dimension of this span.
		
		Now suppose $T \in \Sym_n(F)$ and $\varphi \in S(L^n)$, and set
		\begin{equation}
			Z(T, \varphi)^{\natural} \ := \ \sum_{\substack{ \bfx \in \Omega(T) \\ \text{mod } \Gamma}}  \varphi(\bfx) \cdot Z(\bfx), 
		\end{equation}
		where 
		\begin{equation}
			\Omega(T) \ := \  \{ \bfx = (\bfx_1, \dots, \bfx_n)\in V^n \ | \   \langle \bfx_i, \bfx_j \rangle = T_{ij} \} .
		\end{equation}
		This cycle is rational over $E$. If $Z(T,\varphi)^{\natural} \neq 0$, then $T$ is necessarily totally positive semidefinite, and in this case $Z(T,\varphi)^{\natural}$ has codimension equal to the rank of $T$. 
		
		Finally, we define a $S(L^n)^{\vee}$-valued cycle $Z(T)^{\natural}$ by the rule 
		\begin{equation}
		Z(T)^{\natural} \colon \varphi \mapsto Z(T, \varphi)^{\natural}.
		\end{equation}
		for $\varphi \in S(L^n)$.

	\subsection{The cotautological bundle}

	Let $\calE \to X $ denote the tautological bundle: over the complex points $X(\bbC) = \Gamma \backslash \bbD^+$, the fibre $\calE_z$   at a point $z \in \bbD^+$ is simply the line corresponding to $z$ in the model \eqref{eqn:symmetric space proj model}. There is a natural Hermitian metric $\| \cdot \|^2_{\calE}$ on $\calE(\bbC)$, defined at a point  $z \in \bbD^+$ by the  formula 
$	\| v_z \|^2_{\calE,z} \ = \ - \langle v_z, v_z \rangle$ for $ v_z \in z$.  

Consider the arithmetic class
	\begin{equation}
		\widehat \omega  = - \widehat c_1(\calE, \| \cdot \|_{\calE}) \ \in \ \ChowHat{1}_{\bbC}(X);
	\end{equation}
	concretely, $\widehat\omega = - (\mathrm{div} s, - \log \| s\|_{\calE}^2 ) $, where $s$ is any meromorphic section of $\calE$. Finally, for future use, we set
	\begin{equation}
	\Omega :=   -  c_1(\calE, \| \cdot \|_{\calE}) \in A^{1,1}(X(\bbC))
	\end{equation}
	where  $-\Omega = c_1(\calE, \| \cdot \|_{\calE}) $ is the first Chern form attached to $(\calE, \|\cdot \|_{\calE})$; here the Chern form is normalized as in \cite[\S 4.2]{SouleBook}. Note that $-\Omega$ is a K\"ahler form, cf.\ \cite[\S 2.2]{GarciaSankaran}.
	\begin{remark}
			Elsewhere in the literature, one often finds a different normalization (i.e.\ an overall multiplicative constant) for the metric $\| \cdot \|_{\calE}$ that is better suited to certain arithmetic applications; for example, see  \cite[\S 3.3]{KRYbook}. In our setting, however,  \Cref{remark:Chow vanishing} implies that rescaling the metric does not change the Chern class in $\ChowHat{1}_{\bbC}(X)$. 
	\end{remark}

		\subsection{Green forms and arithmetic cycles} \label{sec:prelim arith cycles}
		In this section, we  sketch the construction of a family of Green forms for the special cycles, following \cite{GarciaSankaran}.
		
		We begin by recalling that for any tuple $x = (x_1, \dots x_n) \in V_1^n = (V \otimes_{\sigma_1, F} \bbR)^n$, Kudla and Millson (see \cite{KudlaMillsonIHES}) have defined a Schwartz form $\varphi_{\KM}(x)$,  which is valued in the space of closed $(n,n)$ forms on $\bbD^+$, and is of exponential decay in $x$. Let $T(x) \in \Sym_{n}(\bbR)$ denote the matrix of inner products, i.e.\ $T(x)_{ij} = \langle x_i, x_j \rangle$, and consider the normalized form
		\begin{equation} \label{eqn:PhiKM normalized}
			\varphi^o_{\KM}(x) \ := \ \varphi_{\KM}(x) \, e^{2 \pi \mathrm{tr} T(x) } .
		\end{equation}

		In \cite[\S 2.2]{GarciaSankaran},  another form  $\nu^o(x)$, valued in the space of smooth $(n-1,n-1)$ forms on  $\bbD^+$ \, is defined (this form is denoted by $\nu^o(x)_{[2n-2]}$ there). It satisfies the relation
		\begin{equation} \label{eqn:nu transgression}
		 	\ddc	\nu^o( \sqrt{u} x) \ = \ - u \frac{\partial}{\partial u}  \varphi_{\KM}^o(\sqrt{u} x), \qquad u \in \bbR_{>0}.
		\end{equation}
		
		For a complex parameter $\rho \gg 0$, let 
		\begin{equation}
			\lie g^o(x;\rho) \ := \ \int_1^{\infty} \nu^o(\sqrt u x) \frac{du}{u^{1+\rho}};
		\end{equation}
		then $\lie g^o(x, \rho)$ defines a smooth form for $Re(\rho) \gg 0$. The corresponding current admits a meromorphic continuation to a neighbourhood of $\rho=0$ and we set
		\begin{equation}
			\lie g^o(x) := \mathop{CT}_{\rho=0} \, \lie g^o(x;\rho).
		\end{equation}
		Note that, for example, 
		\begin{equation} \label{eqn:green current zero}
			\lie g^o(0) \ = \ \nu^o(0) \, \mathop{CT}_{\rho = 0} \int_1^{\infty}  \frac{du}{u^{1+\rho}} = 0. 
		\end{equation}
		
		In general, the current $\lie g^o(x)$ satisfies the equation 
			\begin{equation} \label{eqn:Green equation for g}
				\ddc \lie g^o(x) + \delta_{\bbD^+_{\bfx}} \wedge \Omega^{n - r(x)} = \varphi^o_{\KM}(x)			\end{equation}
		where $r(x) = \dim\mathrm{span}(x) = \dim\mathrm{span}(x_1, \dots, x_n)$; for details regarding all these facts, see \cite[\S 2.6]{GarciaSankaran}.

		Now suppose $T \in \Sym_n(F)$. Following \cite[\S4]{GarciaSankaran}, we define an $S(L^n)^{\vee}$-valued current $\lie g^o(T,\bfv)$, depending on a parameter $\bfv \in \Sym_n(F\otimes_{\bbQ} \bbR)_{\gg 0}$, as follows: let $v = \sigma_1(\bfv)$ and choose any matrix $a \in \GL_n(\bbR)$ such that $v = a a'$. Then $\lie g^o(T,\bfv)$ is defined by the formula
		\begin{equation}
			\lie g^o(T, \bfv)(\varphi) \ := \ \sum_{\bfx \in \Omega(T)}  \varphi(\bfx) \  \lie g^o \left( \sigma_1(\bfx)a \right),  
		\end{equation}
		where $ \sigma_1(\bfx) \in V_1^n$; by \cite[Proposition 2.12]{GarciaSankaran}, this is independent of the choice of $a \in \GL_n(\bbR)$. 
			Note that $\lie g^o(T, \bfv)$ is a $\Gamma$-invariant current on $\bbD^+$ and hence descends to $X(\bbC) = \Gamma \backslash \bbD^+$. 
		
		Next, consider the $S(L^n)^{\vee}$-valued differential form $\omega(T,\bfv)$, defined by the formula
		\begin{equation} \label{eqn:KM theta q coeff}
			\omega(T,\bfv)(\varphi)  \ := \ \sum_{\bfx \in \Omega(T)}  \varphi(x) \, \varphi_{\KM}^o( \sigma_1(\bfx) a), \qquad \sigma_1(\bfv) = aa',
		\end{equation}
		and which  is a $q$-coefficient of the \emph{Kudla-Millson theta series}
		\begin{equation} \label{eqn:Kudla Millson theta}
			\Theta_{\KM}(\bftau) \ = \ \sum_{T \in \Sym_n(F)} \, \omega(T,\bfv) \, q^T, \qquad 
		\end{equation}
		where $ \tau \in \bbH^d_n$, and $ \bfv = \mathrm{Im}(\bftau)$. We then have the equation of currents 
		\begin{equation} \label{eqn:global green eqn}
			\ddc \lie g^o(T,\bfv) \ + \ \delta_{Z(T)(\bbC)} \wedge \Omega^{n - \mathrm{rank}T} \ = \  \omega(T,\bfv)
		\end{equation}
		on $X$,  see \cite[Proposition 4.4]{GarciaSankaran}
		
		In particular, if $T$ is non-degenerate, then $\rank(T) = n$ and  $\lie g^o(T,\bfv)$ is a Green current  for the cycle $Z(T)^{\natural}$; in this case, we obtain an arithmetic special cycle
		\begin{equation}
			\widehat Z(T, \bfv) := (Z(T)^{\natural}, \lie g^o(T,\bfv) ) \ \in \ \ChowHat{n}_{\bbC}(X) \otimes_{\bbC} S(L^n)^{\vee}.
		\end{equation}
		
			Now suppose $T \in \Sym_n(F)$ is arbitrary,  let $r = \rank(T)$, and fix $\varphi \in S(L^n)$.
			 We may choose a  pair $(Z_0, g_0)$ representing the class $\widehat\omega^{n-r} \in \ChowHat{n}_{\bbC}(X)$,
		such that $Z_0$ intersects $Z(T, \varphi)^{\natural}$ properly and $g_0$ has logarithmic type \cite[\S II.2]{SouleBook}. We then define
		\begin{equation}
			\widehat Z(T, \bfv, \varphi) := \left(Z(T, \varphi)^{\natural} \cdot Z_0, \ \lie g^o(T,\bfv, \varphi) + g_0 \wedge \delta_{Z(T,\varphi)(\bbC)} \right) \ \in \ \ChowHat{n}_{\bbC}(X).
		\end{equation}
		The reader may consult \cite[\S 5.4]{GarciaSankaran} for more detail on this construction, including the fact that it is independent of the choice of $(Z_0, g_0)$. 
		
		Finally, we define  a class $\widehat Z(T,\bfv) \in \ChowHat{n}_{\bbC}(X) \otimes S(L^n)^{\vee}$ by the rule
		\begin{equation}
			\widehat Z(T,\bfv)(\varphi) = \widehat Z(T, \bfv, \varphi).
		\end{equation}
				\begin{remark}
			In \cite{GarciaSankaran}, the Green current $\lie g^o(T,\bfv)$ is augmented by an additional term, depending on $\log(\det \bfv)$, when $T$ is degenerate see \cite[Definition 4.5]{GarciaSankaran}. This term was essential in establishing the archimedean arithmetic Siegel-Weil formula in the degenerate case; however, in the setting of the present paper, \Cref{remark:Chow vanishing} implies that  this additional term vanishes upon passing to $\ChowHat{n}_{\bbC}(X)$, and can be omitted from the  discussion without consequence. In particular, according to our definitions, we have 
			\begin{equation}
				\widehat Z(0_n, \bfv)(\varphi) = \varphi(0) \cdot \widehat\omega^{n}.
			\end{equation}
		\end{remark}
	
		\subsection{Hilbert-Jacobi modular forms} \label{sec:modularity def}
	In this section, we briefly review the basic definitions of vector-valued (Hilbert) Jacobi modular forms, mainly to fix notions.  For convenience, we work in ``classical" coordinates and only with parallel scalar weight. Throughout, we fix an integer $n \geq 1.$

	
	We begin by briefly recalling the theory of metaplectic groups and the Weil representation; a convenient summary for the facts mentioned here, in a form useful to us,  is \cite[\S 2]{JiangSoudryGenericityII}. For a place $v \leq \infty$, let $\Sptilde_n(F_v)$ denote the metaplectic group, a two-fold cover of $\Sp_n(F_v)$; as a set, $\Sptilde_n(F_v) = \Sp_n(F_v) \times \{\pm 1\}$.
		When $F_v = \bbR$, the group $\Sptilde_n(\bbR)$ is isomorphic to the group of pairs $(g, \phi)$, where $g = \psm{A & B \\ C & D} \in \Sp_n(\bbR)$ and $\phi \colon \bbH_n \to \bbC$ is a function such that $\phi(\tau)^2 = \det (C \tau + D)$; in this  model, multiplication is given by
			\begin{equation}
		(g, \phi(\tau)) \cdot (g', \phi'(\tau)) = (gg',  \phi(g' \tau)  \phi'(\tau)  ).
		\end{equation}
		
	At a non-dyadic finite place, there exists a canonical embedding $\Sp_n(\calO_v) \to \Sptilde_n(F_v)$. Consider the restricted product $\prod'_{v \leq \infty} \Sptilde_n(F_v)$ with respect to these embeddings; the global double cover  $\Sptilde_{n,\bbA}$   of $\Sp_n(\bbA)$ is the quotient $\prod'_{v \leq \infty} \Sptilde_n(F_v)/ I$ of this restricted direct product by the subgroup 
	\begin{equation}
	 I := \{ (1, \epsilon_v)_{v \leq \infty} \, | \, \prod_v \epsilon_v = 1, \, \epsilon_v = 1 \text{ for almost all } v  \}.
	 \end{equation}
	Moreover, there is a  splitting  		
		\begin{equation}\label{eqn:metaplectic canonical splitting}
		\iota_F \colon \Sp_n(F) \hookrightarrow \Sptilde_{n,\bbA}, \qquad \gamma \mapsto \prod_v (\gamma, 1)_v \cdot I.
		\end{equation}	
	
	Let $\widetilde \Gamma'$ denote the full inverse image of $\Sp_{n}(\calO_F)$ under the covering map $\prod_{v | \infty} \Sptilde_n(F_v) \to \Sp_n(F \otimes_{\bbQ} \bbR)$. We obtain an action $\rho$ of $\widetilde \Gamma'$ on the space $ S(V(\bbA_f)^n)$ as follows. Let $\omega $ denote the\footnote{Here we take the Weil representation for the standard additive character $\psi_F \colon \bbA_F / F \to \bbC$, which we suppress from the notation.} Weil representation  of $\Sptilde_{n,\bbA}$ on $S(V(\bbA)^n)$. Given $\widetilde \gamma \in \widetilde \Gamma'$, choose $\widetilde \gamma_f \in \prod'_{v < \infty} \Sptilde_n(F_v)$ such that $\widetilde \gamma \widetilde \gamma_f \in \mathrm{im}(\iota_F)$ and set 
	\begin{equation}
		\rho(\widetilde \gamma) \ := \ \omega(\widetilde \gamma_f).
	\end{equation}
	Recall that we had fixed a lattice $L \subset V$. The subspace $S(L^n) \subset S(V(\bbA_f)^n)$, as defined in  \eqref{eqn:intro S(L) defn}, is stable under the action of $\widetilde{\Gamma}'$; when we wish to emphasize this lattice, we denote the corresponding action by $\rho_L$.

	 For a half-integer  $\kappa \in \frac12 \mathbb Z$, we define a (parallel, scalar) weight $\kappa$ slash operator, for the group $\widetilde \Gamma'$ acting on the space of functions $f  \colon \bbH^d_n \to S(L^n)^{\vee}$, by the formula
	\begin{equation} \label{eqn:slash operators}
	f|_{\kappa}[\widetilde \gamma] (\bftau) \ = \  \prod_{v|\infty} \phi_v(\sigma_v(\bftau))^{- 2 \kappa} \rho_L^{\vee}(\widetilde \gamma^{-1})\cdot f( g \bftau), \qquad \widetilde \gamma = (g_v, \phi_v(\tau))_{v | \infty}
	\end{equation}
	where $g = (g_v)_v$. 
	
	If $n >1$, consider the Jacobi group $G^J = G^J_{n,n-1} \subset \Sp_{n}$; for any ring $R$, its $R$-points are given by
	\begin{equation}
		G^J(R) := 
							\left\{   g =  
										\left( \begin{array}{cc|cc} 
													a & 0 & b & a\mu - b \lambda   \\
													\lambda^t & 1_{n-1} & \mu^t & 0 \\
													\hline
													c & 0 & d & c\mu - d \lambda \\ 
													0&0&0&1_{n-1}
										\end{array}  \right) \ | \ \begin{pmatrix} a&b\\c&d\end{pmatrix} \in \SL_2(R) , \ \mu, \tau \in M_{1,n-1}(R)\right\}.
	\end{equation} 
	Define $\widetilde \Gamma^J \subset \widetilde \Gamma'$ to be the inverse image of $G^J(\calO_F)$ in $\widetilde G'_{\bbR}$.

	\begin{definition} \label{def:Jacobi form}
		Suppose 
		\begin{equation}
			f \colon \bbH_n^d \to S(L^n)^{\vee}
		\end{equation}
		is a smooth function. 
		Given $T_2 \in \Sym_{n-1}(F)$, we say that $f(\bftau)$ transforms like a Jacobi modular form of genus $n$, weight $\kappa$ and index $T_2$ if the following conditions hold. 
			\begin{enumerate}[(a)]
					\item For all $\bfu_2 \in \Sym_{n-1}(F_{\bbR})$, 
						\begin{equation}
							f \left( \bftau + \psm{ 0 & \\ & \bfu_2 } \right)  \ =  \ e(  T_2 \bfu_2)  f(\bftau).
						\end{equation}
					\item  For all $\widetilde \gamma \in \widetilde \Gamma^J$,
						\begin{equation}
								f|_{\kappa}[\widetilde \gamma] (\bftau) \ = \ f(\bftau).
						\end{equation}

			\end{enumerate}

		Let $A_{\kappa,T_2}(\rho_L^{\vee})$ denote the space of $S(L^n)^{\vee}$-valued smooth functions that transform like a Jacobi modular form of weight $\kappa$ and index $T_2$. 
	\end{definition}
	
	\begin{remark}
		\begin{enumerate}
			\item If desired, one can impose further analytic properties of $f$ (holomorphic, real analytic, etc.).
			\item If $n = 1$, then we simply say that a function $f \colon \bbH^d_1 \to S(L)^{\vee}$ transforms like a (Hilbert) modular form of weight $\kappa$ if it satisfies $f|_{\kappa}[ \tilde \gamma](\bftau)  = f(\bftau)$ as usual.
			\item An $S(L^n)^{\vee}$-valued Jacobi modular form $f$, in the above sense, has a Fourier expansion of the form
			\begin{equation} \label{eqn:Fourier expansion of Jacobi form}
			f (\bftau) \ := \ \sum_{T = \psm{ * & * \\ * & T_2}} c_{ f}(T, \bfv)  \,q^T
			\end{equation}
			where the coefficients $c_f(T,\bfv)$ are smooth functions $ c_f(T,\bfv) \colon \Sym_n(F\otimes_{\bbQ}{\bbR})_{\gg 0} \to S(V^n)^{\vee}$. The dependence on $\bfv$ arises from the natural expectation that the Fourier-Jacobi coefficients of non-holomorphic Siegel modular forms should be Jacobi forms. 
		 \end{enumerate}
	\end{remark}

We now clarify what it should mean for  generating series with coefficients in arithmetic Chow groups, such as those appearing in \Cref{thm:Intro main theorem}, to be modular.  

First, let $D^{n-1}(X)$ denote the space of currents on $X(\bbC)$ of complex bidegree $(n-1,n-1)$, and note that there is a map
\begin{equation}
a \colon D^{n-1}(X) \ \to  \ \ChowHat{n}_{\bbC}(X,\Dcur), \qquad a(g) = ( 0 , \, [\ddc g, g]).
\end{equation}

\begin{definition} \label{def:smooth currents} Define the space $A_{\kappa, T_2}( \rho^{\vee}; D^{n-1}(X))$ of ``Jacobi forms valued in $S(L^n)^{\vee} \otimes_{\bbC} D^{n-1}(X)$" as the space of functions
\begin{equation}
\xi \colon \bbH^d_n \to D^{n-1}(X)\otimes_{\bbC} S(L^n)^{\vee}
\end{equation} 
such that the following two conditions hold.
	\begin{enumerate}[(a)]
		\item For every smooth form $\alpha$ on $X$, the function $\xi(\bftau)(\alpha)$ is an element of $A_{\kappa, T_2}( \rho_L^{\vee})$, and in particular, is smooth in the variable  $\bftau$. 
		\item Fix an integer $k\geq 0$ and let $\| \cdot \|_k$ be an algebra seminorm, on the space of smooth differential forms on $X$, such that given a sequence $\{ \alpha_i \} $, we have $\| \alpha_i \|_k \to 0$ if and  only if  $\alpha_i$, together with all partial derivatives of order $\leq k$, tends uniformly to zero. We then require that for every  compact  subset $ C\subset \bbH^d_n$, there exists a constant $c_{k,V}$ such that
			\begin{equation}
				|\xi(\bftau)(\alpha)| \leq c_{k,C} \| \alpha\|_k
			\end{equation}
			for all $\bftau \in  C$ and all smooth forms $\alpha$. 
	\end{enumerate}
\end{definition}
The second condition ensures that any such function admits a Fourier expansion as in \eqref{eqn:Fourier expansion of Jacobi form} whose coefficients  are continuous in the sense of distributions, i.e.\ they are again $S(L^n)^{\vee}$-valued currents.  

\begin{definition} \label{definition of modularity}
	Given a collection of classes $\widehat Y(T,\bfv) \in \ChowHat{n}(X,\Dcur) \otimes_{\bbC}S(L^n)^{\vee}$, consider the formal generating series 
	\begin{equation}
	\widehat \Phi_{T_2}(\bftau) \ := \  \sum_{\substack{T = \psm{ * & * \\ * & T_2}}} \widehat Y(T, \bfv)  \, q^T. 
	\end{equation}
	Roughly speaking, we say that $\widehat\Phi_{T_2}(\bftau)$ is modular (of weight $\kappa$ and index $T_2$) if there is an element 
	\begin{equation} \label{eqn:modularity def spaces}
		\widehat \phi(\bftau) \in A_{\kappa, T_2}(\rho_L^{\vee})\otimes_{\bbC} \ChowHat{n}_{\bbC}(X,\Dcur)  \ + \ a\left( A_{\kappa, T_2}(\rho_L^{\vee} ; D^{n-1}(X)) \right)
	\end{equation}
	whose Fourier expansion coincides with $\widehat \Phi_{T_2}(\bftau)$. More precisely, we define the modularity of $\widehat\Phi_{T_2}(\bftau)$ to mean that there are finitely many classes 
	\begin{equation}
		\widehat Z_1, \dots, \widehat Z_r \in \ChowHat{n}(X,\Dcur)
	\end{equation} 
	and Jacobi forms
	\begin{equation}
		f_1, \dots f_r \in A_{\kappa,T_2}(\rho^{\vee}), \qquad  g \in A_{\kappa,T_2}(\rho^{\vee}; D^{n-1}(X))
	\end{equation}
	such that 
	\begin{align} \label{eqn:def of modularity cycle decomp}
		\widehat Y(T, \bfv) =& \sum_i  c_{f_i}(T, \bfv) \, \widehat Z_i +  a \left( c_{g}(T, \bfv)    \right) \ \in \ \ChowHat{n}_{\bbC}(X,\Dcur) \otimes_{\bbC} S(L^n)^{\vee}
	\end{align}
	for all $T = \psm{* & * \\ * & T_2}$ . 
	
\end{definition}
\begin{remark} \label{remark:modularity Chow}
	\begin{enumerate}
		\item If $\widehat Z_1, \dots, \widehat Z_r \in \ChowHat{n}_{\bbC}(X)$ and $g(\bftau)$ takes values in  the space of (currents represented by) smooth differential forms on $X$, then we say that $\widehat \Phi_{T_2}(\bftau)$ is valued in $\ChowHat{n}_{\bbC}(X) \otimes S(L^n)^{\vee}$; indeed, in this case, the right hand side of  \eqref{eqn:def of modularity cycle decomp} lands in this latter group.
		\item As before, one may also impose additional analytic conditions on the forms $f_i, g$ appearing above if desired.
		\item 
			Elsewhere in the literature (e.g.\  \cite{BorcherdsGKZ,Bruinier-totally-real,ZhangThesis}), one finds a notion of modularity that amounts to omitting the second term in \eqref{eqn:modularity def spaces}; this notion is well-adapted to the case that the generating series of interest are holomorphic, i.e.\ the coefficients are independent of the imaginary part of $\bftau$. 
			
			In contrast, the generating series  that figure in our main theorem depend on these parameters in an essential way. Indeed, the Green forms $\lie g^o(T,\bfv)$ vary smoothly in $\bfv$; however, to the best of the author's knowledge, there is no natural topology on $\ChowHat{n}(X)$, or $\ChowHat{n}(X, \Dcur)$,  for which the corresponding family $\widehat Z(T,\bfv)$ varies smoothly in $v$. As we will see in the course of the proof of the main theorem, the additional term in \eqref{eqn:modularity def spaces} will allow us enough flexibility to reflect the non-holomorphic behaviour of the generating series. Similar considerations appear in \cite{EhlenSankaran} in the codimension one case.  

	\end{enumerate}
\end{remark}


	\section{The genus one case}
		In this section, we prove the main theorem in the case $n =1$; later on, this will be a key step in the proof for general $n$. 
	The proof of this theorem amounts to a comparison with a generating series of special divisors equipped with a different family of Green functions, defined by Bruinier. A similar comparison appears in \cite{EhlenSankaran} for unitary groups over imaginary quadratic fields; in the case at hand, however, the compactness  of $X$ allows us to apply spectral theory and  simplify the argument considerably.

	Suppose $t \gg 0$. In \cite{Bruinier-totally-real}, Bruinier constructs an $S(L)^{\vee}$-valued Green function  $\Phi(t)$  for the divisor $Z(t) = Z(t)^{\natural}$. 
		To be a bit more precise about this, recall the Kudla-Millson  theta function  $\Theta_{\KM}(\bftau)$  from \eqref{eqn:Kudla Millson theta}. As a function of $\bftau$, the theta function $\Theta_{\KM}$ is non-holomorphic and transforms as a Hilbert modular form of parallel weight $\kappa = p/2 + 1$. It is moreover of moderate growth, \cite[Prop. 3.4]{Bruinier-totally-real} and hence can be paired, via the Petersson pairing, with cusp forms.
		Let $\Lambda_{\KM}(\bftau) \in S_{\kappa}(\rho_L)$ denote the cuspidal projection, defined by the property
	\begin{equation}
		\langle \Theta_{\KM}, g \rangle^{\mathrm{Pet}}  = \langle \Lambda_{\KM}, g\rangle^{\mathrm{Pet}}
	\end{equation}
	for all cusp forms $g \in S_{\kappa}(\rho)$. 
	
	Writing the Fourier expansion
	\begin{equation}
		\Lambda_{\KM}(\bftau) = \sum_t  c_{\Lambda}(t)   \ q^t, \qquad \Lambda_{\KM,t} \in A^{1,1}(X) \otimes S(L)^{\vee}
	\end{equation}
	it follows from \cite[Corollary 5.16, Theorem 6.4]{Bruinier-totally-real} that  $\Phi(t)$ satisfies the equation
	\begin{equation} \label{eqn:Bruinier Green equation}
	\ddc [\Phi(t)] \ + \ \delta_{Z(t)} \ = \ \left[ c_{\Lambda}(t) + B(t) \cdot \Omega \right]
	\end{equation}
	of currents on $X$,	where
	\begin{equation}
	B(t) \ := \ - \frac{ \deg(Z(t))}{\mathrm{vol}(X, (-\Omega)^p) } \ \in \ S(L)^{\vee}.
	\end{equation}
	recall here that $(-\Omega)^p$ induces a volume form on $X$.

Finally, define classes $	\widehat Z_{\mathrm{Br}}(t) \in \ChowHat{1}_{\bbC}(X) \otimes S(L)^{\vee}$ as follows:
	\begin{equation}
		\widehat Z_{\mathrm{Br}}(t)  \ = \ 
			\begin{cases} \left( Z(t),  \Phi(t) \right),   & \text{if } t \gg 0 \\   \widehat{ \omega} \otimes \ev_0, & \text{if } t = 0 .  \\ 0, & \text{otherwise,} \end{cases}
	\end{equation}
where $\ev_0 \in S(L)^{\vee}$ is the functional $\varphi \mapsto \varphi(0)$. 

We then have the generating series
 \begin{equation}
 \widehat\phi_{\mathrm{Br}}(\tau) = \sum_t \widehat Z_{\mathrm{Br}}(t) \, q^t.
 \end{equation}

\begin{theorem}[Bruinier] \label{Bruinier modularity gen fibre}
	The generating series $\widehat\phi_{\mathrm{Br}}(\tau)$ is a (holomorphic) Hilbert modular form of parallel weight $\kappa = p/2+1$. More precisely, there are finitely many classes $\widehat Z_1, \dots \widehat Z_r \in \ChowHat{1}_{\bbC}(X)$ and holomorphic Hilbert modular forms $f_1, \dots, f_r$ such that $\widehat Z_{\mathrm{Br}}(t) = \sum c_{f_i}(t) \, \widehat Z_i$ for all $t\in F$.

	\begin{proof} The proof follows the same argument as  \cite[Theorem 7.1]{Bruinier-totally-real}, whose main steps we recall here.  Bruinier defines a space $M^!_{k}(\rho_L)$ of weakly holomorphic forms \cite[\S 4]{Bruinier-totally-real} of a certain``dual" weight $k $; each $f \in M^!_{k}(\rho_L)$ is defined by a finite collection of vectors $c_f(m) \in S(L)^{\vee}$ indexed by $m \in F$. Applying Bruinier's criterion for the modularity of a generating series, cf.\  \cite[(7.1)]{Bruinier-totally-real}, we need to show that
		\begin{equation}
				\sum_m c_f(m) \widehat Z_{\mathrm{Br}}(m) = 0 \in \ChowHat{1}(X)
		\end{equation}
		for all $f \in M_k^!(\rho_L)$. Let $c_0 = c_f(0)(0)$, and assume $c_0 \in \mathbb Z$. By \cite[Theorem 6.8]{Bruinier-totally-real}, after replacing $f$ by a sufficiently large integer multiple, there exists an analytic meromorphic section $\Psi^{an} $ of $(\omega^{an})^{-c_0}$ such that
		\begin{equation}
			\mathrm{div}  \, \Psi^{an} = \sum_{m \neq 0} c_f(m) \cdot Z(m)^{an}.
		\end{equation}
		and with 
		\begin{equation}
			\mathrm - \log \| \Psi^{an} \|^2 = \sum_{m \neq 0} c_f(m) \cdot \Phi(m).
		\end{equation}
		
		Recall  that $X$ is projective; by GAGA and the fact that the $Z(m)$'s are defined over $E$, there is an $E$-rational section $\psi$ of $\omega^{-c_0}$ and a constant $C \in \bbC$ such that 
		\begin{equation}
			\mathrm{div} (\psi) = \sum_{m\neq 0} c_f(m) Z(m), \qquad - \log \| \psi^{an} \|^2 =  - \log \| \Psi^{an} \| ^2 + C.
		\end{equation}
		Thus
		\begin{equation}
			-c_0 \cdot \widehat \omega  \ = \ \widehat{\mathrm{div}}( \psi) = \sum_{m \neq 0} c_f(m) \cdot \widehat Z_{\mathrm{Br}}(m)  \ + \ (0, C) \ \in \ \ChowHat{1}(X).
		\end{equation}
		However, as in  \Cref{remark:Chow vanishing}, the class $(0, C) = 0$, and thus we find 
		\begin{equation}
			\sum_m c_f(m) \widehat Z_{\mathrm{Br}}(m) \ =\   c_0 \cdot \widehat \omega + \sum_{m \neq 0 } c_f(m) \widehat Z_{\mathrm{Br}}(m) \ = \ 0
		\end{equation}
		as required. 
	\end{proof}
\end{theorem}

Now we consider the difference 
\begin{equation}
	\widehat\phi_1(\bftau) - \widehat \phi_{\mathrm{Br}}(\bftau)  = \sum_t (0, \lie g^o(t,\bfv) - \Phi(t))  \, q^t
\end{equation}
whose terms are classes represented by purely archimedean cycles. Comparing the Green equations \eqref{eqn:Green equation for g} and \eqref{eqn:Bruinier Green equation}, we have that for $t \neq 0$ and any smooth form  $\eta$,
\begin{equation}
	\ddc [\lie g^o(t,v) - \Phi(t)] (\eta) = \int_X \left( \lie g^o(t,\bfv) - \Phi(t) \right) \, \ddc \eta = \int_X \left( \omega(t,\bfv) - c_{\Lambda}(t) - B(t) \Omega \right) \wedge \eta
\end{equation}
where $\omega(t,\bfv)$ is the $t$'th $q$-coefficient of $\Theta_{\KM}(\bftau)$; in particular, elliptic regularity implies that the difference $\lie g^o(t,\bfv) - \Phi(t)$ is smooth on $X(\bbC)$.

\begin{theorem} \label{genus one diff modularity} There exists a smooth $S(L)^{\vee}$-valued function $s(\bftau,z) $ on $\bbH_1^d \times X(\bbC)$  such that the following holds.
	\begin{enumerate}[(i)]
		\item For each fixed $z \in X(\bbC)$, the function $s(\bftau,z)$ transforms like a Hilbert modular form in $\bftau$.
		\item Let 
		\begin{equation}
		s(\bftau,z) = \sum_t c_s(t,\bfv,z) \ q^t
		\end{equation}
		denote its $q$-expansion in $\bftau$; then for each $t$, we have
		\begin{equation}
		(0, \lie g^o(t,\bfv) -  \Phi(t) ) \ = \  (0, c_s(t, \bfv,z)) \ \in \ \ChowHat{1}_{\bbC}(X) \otimes_{\bbC}S(L)^{\vee}
		\end{equation}
	\end{enumerate}
\end{theorem}
Combining this theorem with \Cref{Bruinier modularity gen fibre}, we obtain:
\begin{corollary} \label{genus one modularity}
	The generating series $\widehat \phi_1(\bftau) $ is modular, valued in $\ChowHat{1}_{\bbC}(X) \otimes S(L)^{\vee}$, in the sense of \Cref{remark:modularity Chow}(i). \qed
\end{corollary}
	
	\begin{proof}[Proof of \Cref{genus one diff modularity}]
	  Recall that the  $(1,1)$ form $- \Omega$ is a K\"ahler form on $X$. Let $- \Delta_X$ denote the corresponding Laplacian; the eigenvalues of $- \Delta_X$  are non-negative, discrete in $\bbR_{\geq 0}$, and each eigenspace is finite dimensional.

		Write $\Delta_X = 2 (\partial \partial^* +  \partial^* \partial)$ and let $L \colon \eta \mapsto -\eta \wedge (-\Omega)$ denote the Lefschetz operator. 
		From the K\"ahler identities $[L, \partial] = [L,\overline \partial] = [L, \Delta_S] = 0$ and $[L, \partial^*] =  i \overline \partial$, an easy induction argument shows that
		\begin{equation}
		\partial^* \circ L^k \ =\ L^k \circ \partial^*  \ -  \  i k  \, \overline \partial \circ L^{k-1}
		\end{equation}
		for $k \geq 1$. 
		
		Thus, for a smooth function $\phi$ on $X$, we have
		\begin{align}
		\Delta_X (\phi) \cdot (-\Omega)^p   = \Delta_X \circ L^p(\phi) &= 2 \partial \partial ^* \circ L^p (\phi) \\ &= \ 2 \ \partial \circ \left( L^p \circ \partial^* - i  p \overline \partial  \circ L^{p-1} \right)(\phi ) \\
		&= \ -   2 i  p \ \partial \overline \partial  \left( \phi \wedge (-\Omega)^{p-1} \right) \\
		&= \ - 4  \pi p \ \ddc \left( \phi \wedge (-\Omega)^{p-1} \right);
		\end{align}
		note here that $p = \dim_{\bbC} (X).$
		
		Consider the Hodge pairing
		\begin{equation}
			\langle f,g \rangle_{L^2} = \int_X f \,  \overline g  \, (-\Omega)^p = (-1)^p\int_X f \,  \overline g  \, \Omega^p.
		\end{equation}
		If $\lambda >0$ and  $\phi_{\lambda}$ is a Laplace eigenfunction, we have that for any $t\neq0$, 
		\begin{align}
			\langle \lie g^o(t,\bfv) - \Phi(t), \phi_{\lambda}\rangle_{L^2} &= \lambda^{-1} \langle \lie g^o(t,\bfv) - \Phi(t), - \Delta_X \phi_{\lambda} \rangle_{L^2}  \\
			&= (-1)^p  \lambda^{-1} \int_X ( \lie g^o(t,\bfv) - \Phi(t)) \cdot (\overline{- \Delta_X \phi_{\lambda}}) 	\cdot \Omega^p \\
			&= (-1)^{p+1} \frac{ 4 \pi p }{\lambda} \int_X \left( \lie g^o(t,\bfv) - \Phi(t) \right) \cdot \ddc \left( \overline \phi_{\lambda} \, \Omega^{p-1} \right)  \\
			&= (-1)^{p+1 } \frac{ 4 \pi p }{\lambda} \int_X \left( \omega(t,\bfv) - c_{\Lambda}(t)   - B(t) \Omega \right) \wedge \overline \phi_{\lambda} \Omega^{p-1}
		\end{align}
		Note that $\int_X \overline \phi_{\lambda} \Omega^p = \langle 1, \phi_{\lambda} \rangle_{L^2} = 0$, as $\lambda >0$ and so $\phi_{\lambda}$ is orthogonal to constants; thus the term involving $B(m) \Omega$ vanishes, and so
		\begin{equation}
				\langle \lie g^o(t,\bfv) - \Phi(t), \phi_{\lambda}\rangle_{L^2} =   (-1)^{p+1 }\frac{ 4 \pi p }{\lambda} \int_X \left( \omega(t,\bfv) - c_{\Lambda}(t)  \right) \wedge \overline \phi_{\lambda} \Omega^{p-1}
		\end{equation}
		for all $t \neq 0$. This equality also holds for $t=0$, as both sides of this equation vanish. Indeed, for the left hand side we have $\lie g^o(0,\bfv) = 0$, cf.\ \eqref{eqn:green current zero}, and $\Phi(0) = 0$ by definition; on the right hand side, $c_{\Lambda}(0) =0 $ as $\Lambda_{\KM}(\bftau)$ is cuspidal, and the constant term of the Kudla-Millson theta function is given by
		\begin{equation}
			\omega(0, \bfv) =  \Omega \otimes \ev_0.
		\end{equation}

		Now define 
		\begin{equation}
		h(\bftau,z) = (L^*)^{p-1} \circ \ast  \left( \Theta_{\KM}(\bftau) - \Lambda_{\KM}(\bftau) \right)
		\end{equation}
		where $\ast$ is the Hodge star operator, and $L^*$ is the adjoint of the Lefschetz map $L$. Then  $h(\bftau,z)$ is smooth, and transforms like a modular form in $\bftau$, since both $\Theta_{\KM}(\bftau)$ and $\Lambda_{\KM}(\bftau)$ do; writing its Fourier expansion 
		\begin{equation}
		h(\bftau,z) \ =\ \sum_t c_h(t, \bfv,z) \, q^t,
		\end{equation} 
		we have
		\begin{equation}
		\langle c_h(t, \bfv,z), \phi  \rangle_{L^2} =  (-1)^{p-1} \int_X  \left( \omega(t,\bfv) -c_{ \Lambda}(t)  \right)  \wedge \overline{\phi } \,   \Omega^{p-1} 
		\end{equation}
		for any smooth function $\phi$.
		
		Note that for any integer $N$ and $L^2$ normalized eigenfunction $\phi_{\lambda}$ with $\lambda \neq 0$,  
		\begin{equation} \label{eqn:laplace eigenfunction}
		\langle	h , \phi_{\lambda} \rangle_{L^2} = \lambda^{-N} \langle - \Delta^N_X (h), \varphi_{\lambda} \rangle \leq  \lambda^{-N}  \| - \Delta^N_X(h) \|^2_{L^2}.
		\end{equation}
		Choose an orthonormal basis $\{ \phi_{\lambda}\}$ of $L^2(X)$ consisting of eigenfunctions, and consider the sum  
		\begin{equation}  \label{eqn:s(tau,z) def}
		s(\tau,z) =  4 \pi p \sum_{\lambda >0} \lambda^{-1} \langle h, \phi_{\lambda} \rangle_{L^2} \, \phi_{\lambda};
		\end{equation}
		by Weyl's law, there are positive constants $C_1$ and $C_2$ such that 
		\begin{equation}
		\# \{  \lambda \ | \ \lambda < x \} \sim x^{C_1}
		\end{equation}
		and $\| \phi_{\lambda}\|_{L^{\infty}} = O(\lambda^{C_2})$. Thus taking $N$ sufficiently large in \eqref{eqn:laplace eigenfunction}, we conclude that the sum \eqref{eqn:s(tau,z) def} converges uniformly, and hence defines a smooth function in $(\bftau,z)$. 
		Writing its Fourier expanison as 
		\begin{equation}
		s(\bftau,z) = \sum_t c_s(t,\bfv,z) q^t,
		\end{equation}				
		we have
		\begin{equation}
		\langle c_s(t,\bfv,z) , \phi_{\lambda} \rangle_{L^2} = \langle \lie g^o(t,\bfv) -   \Phi(t), \phi_{\lambda} \rangle_{L^2}
		\end{equation}
		for any eigenfunction $\phi_{\lambda}$ with $\lambda \neq 0$. Thus $ c_s(t,\bfv,z) $ and $ g^o(t,\bfv) -   \Phi(t)$ differ by a  function that is constant in $z$; as $(0,1) = 0 \in \ChowHat{1}_{\bbC}(X)$, we have
		\begin{equation}
		(0, \lie g^o(t,\bfv) -  \Phi(t)) = (0, c_s(t,\bfv,z)) \in \ChowHat{1}_{\bbC}(X) \otimes_{\bbC} S(L)^{\vee},
		\end{equation}
		which concludes the proof of the theorem.

	\end{proof}

	\section{Decomposing Green currents}
	We now suppose $n > 1$ and fix $T_2 \in \Sym_{n-1}(F)$.  
	
	The aim of this section is to  establish  a decomposition $	\widehat Z(T, \bfv) =  \widehat A(T,\bfv) + \widehat B(T,\bfv)$, where $T = \psm{*&*\\ *& T_2}$.
	Our first step is to decompose Green forms in a useful way; the result can be seen as an extension of the star product formula \cite[Theorem 4.10]{GarciaSankaran} to the degenerate case.
	
	 Let $x = (x_1, \dots, x_n) \in (V_1 )^n =(  V\otimes_{F, \sigma_1} \bbR)^n$ and set $y = (x_2, \dots, x_n) \in V_1^{n-1}$ . By \cite[Proposition 2.6.(a)]{GarciaSankaran}, we may decompose 
		\begin{equation} \label{eqn: go decomposition}
			\lie g^o(x, \rho) =  \int_1^{\infty}  \nu^o(\sqrt t \, x_1) \wedge \varphi^o_{\KM}(\sqrt t \, y) \ \frac{dt}{t^{1+\rho}} \ + \  \int_1^{\infty}  \varphi_{\KM}^o(\sqrt t \, x_1) \wedge 	\nu^o(\sqrt t \, y) \ \frac{dt}{t^{1+\rho}} 
		\end{equation}
	for $Re(\rho) \gg 0$.
	
	By the transgression formula \eqref{eqn:nu transgression},  we may rewrite the second term in \eqref{eqn: go decomposition} as 
		\begin{align}  
			\int_1^{\infty} & \varphi_{\KM}^o  (\sqrt t \, x_1)  \wedge 	\nu^o(\sqrt t \, y) \ \frac{dt}{t^{1+\rho}}    \\
			&= \int_1^{\infty}  \left( \int_1^t \, \frac{\partial}{\partial u} \varphi_{\KM}^o(\sqrt u \, x_1) \, du\right) \wedge \nu^o(\sqrt t \, y) \ \frac{dt}{t^{1+\rho}}    + \varphi_{\KM}^o(x_1) \wedge \int_1^{\infty} \nu^o(\sqrt t \, y)  \frac{dt}{t^{1+\rho}}  \\
			&= \int_1^{\infty} \left( \int_1^t \, - \ddc \nu^o(\sqrt u \, x_1) \, \frac{du}{u}\right) \wedge \nu^o(\sqrt t \, y) \ \frac{dt}{t^{1+\rho}} + \varphi_{\KM}^o(x_1) \wedge \lie g^o(y,\rho) . \label{eqn: g0 second piece}
		\end{align}
	For $t>1$, define  smooth forms
		\begin{equation}
			\alpha_t(x_1,y) :=  \int_1^{t}  \overline{\partial} \nu^o(\sqrt u \, x_1) \, \frac{du}{u} \wedge \nu^o(\sqrt{t} \, y) 
		\end{equation}
	and
		\begin{equation}
			\beta_t(x_1,y) :=  \int_1^t \nu^o(\sqrt{u} \, x_1) \,  \frac{du}{u} \wedge \partial \nu^o(\sqrt{t} y) 
		\end{equation}
	so that 
		\begin{align}
		\eqref{eqn: g0 second piece}
			= \frac{i}{2 \pi } \int_1^{\infty} \partial & \alpha_t(x_1,y) + \overline{\partial} \beta_t(x_1,y) \, \frac{dt}{t^{1+\rho}}  \ - \ \int_1^{\infty}   \left[ \int_1^{t} \nu^o(\sqrt{u} \, x_1) \, \frac{du}{u} \right] \wedge \ddc \nu^o(\sqrt t \, y) \, \frac{dt}{t^{1+ \rho}}  \notag \\ 
				&+ \varphi_{\KM}^o(x_1) \wedge \lie g^o(y, \rho).
		\end{align}
	Finally, we consider the second integral above;  as $Re(\rho)$ is large, we may interchange the order of integration and obtain
	\begin{align}
	\int_1^{\infty}  & \left(\int_1^{t} \nu^o(\sqrt{u} \, x_1) \, \frac{du}{u} \right) \wedge \ddc \nu^o(\sqrt t \, y) \, \frac{dt}{t^{1+ \rho}} \\
	& = \int_1^{\infty} \nu^o(\sqrt u x_1) \wedge  \left( \int_u^{\infty}  \ddc \nu^o(\sqrt t y) \frac{dt}{t^{1+\rho}}   \right) \frac{du}{u} \\
	&=  \int_1^{\infty} \nu^o(\sqrt u x_1) \wedge  \left( \int_u^{\infty}  - \frac{\partial}{\partial t} \varphi_{\KM}^o(\sqrt t y) \frac{dt}{t^{\rho}}   \right) \frac{du}{u} \\
	&= \int_1^{\infty} \nu^o(\sqrt u \, x_1) \wedge \varphi_{\KM}^o(\sqrt{u} \, y) \, \frac{du}{u^{1+ \rho}} - \rho  \int_1^{\infty} \nu^o(\sqrt u\, x_1) \wedge \left( \int_u^{\infty} \varphi_{\KM}^o(\sqrt t \, y) \frac{dt}{t^{1+\rho}} \right) \, \frac{du}{u}
	\end{align}
	Note that the first term here coincides with the first term in \eqref{eqn: go decomposition}. Combining these computations, it follows that
	\begin{align}
	\lie g^o(x,\rho) &= \varphi^o_{\KM}(x_1) \wedge \lie g^o(y,\rho) + \frac{i}{2 \pi } \int_1^{\infty} \partial \alpha_t(x_1,y) + \overline{\partial} \beta_t(x_1,y) \, \frac{dt}{t^{1+\rho}}   \notag  \\
	& 		\qquad \qquad + \rho  \int_1^{\infty} \nu^o(\sqrt u\, x_1) \wedge \left( \int_u^{\infty} \varphi_{\KM}^o(\sqrt t \, y) \frac{dt}{t^{1+\rho}} \right) \, \frac{du}{u}.
	\end{align}
	This identity holds for arbitrary $x =(x_1,y) \in V_1^n$ and $Re(\rho) \gg 0$, and is an identity of smooth differential forms on $\bbD$.
	
	To continue, we view the above line as an identity of currents, and consider meromorphic continuation.\footnote{More precisely, we mean that for every smooth form $\alpha$, the function $[\lie g^o(x,\rho)](\alpha)  = \int_X \lie g^o(x, \rho) \wedge \alpha$ admits a meromorphic continuation in $\rho$, such that the Laurent coefficients are continuous in $\alpha$ in the sense of currents.}   Note that (as currents)
	\begin{align}
 	\rho	\int_1^{\infty} & \nu^o(\sqrt u\, x_1) \wedge \left( \int_u^{\infty} \varphi_{\KM}^o(\sqrt t \, y) \frac{dt}{t^{1+\rho}} \right) \, \frac{du}{u}  \\ 
	&=  \rho \int_1^{\infty} \nu^o(\sqrt u \, x_1) \wedge  \int_u^{\infty} \left( \varphi_{\KM}^o(\sqrt t \, y) - \delta_{\bbD^+_y} \wedge \Omega^{n-1-r(y)}  \right) \frac{dt}{t^{1+\rho}}	 \frac{du}{u} \notag 	 \\
	& \qquad +  \int_1^{\infty} \nu^o(\sqrt u x_1) \wedge \delta_{\bbD^+_y } \wedge \Omega^{n-1-r(y)}  \ \frac{du}{u^{1+\rho}} \
	\end{align}
	where $r(y) = \dim \, \mathrm{span} (y)$. The first term vanishes at $\rho = 0$; indeed, the double integral in the first term is holomorphic at $\rho = 0$, as can easily seen by  by Bismut's asymptotic  \cite[Theorem 3.2]{BismutInv90}
	\begin{equation}
	\varphi_{\KM}^o(\sqrt{t} y) - \delta_{\bbD^+_y} \wedge \Omega^{n-1-r(y)} \ =\  O(t^{-1/2})
	\end{equation}
	as $t \to \infty$.
	
	Next, let 
	\begin{equation}
	\alpha(x_1,y; \rho) \ := \ \int_1^{\infty} \, \alpha_t(x_1,y) \frac{dt}{t^{1+ \rho}}, \qquad \qquad 	\beta(x_1,y; \rho) \ := \ \int_1^{\infty} \, \beta_t(x,y) \frac{dt}{t^{1+ \rho}}.
	\end{equation}
	A straightforward modification of the proof of \cite[Proposition 2.12.(iii)]{GarciaSankaran} can be used to show that $\alpha(x_1, y; \rho)$ and $\beta(x_1, y; \rho)$ have meromorphic extensions, as currents, to a neighbourhood of $\rho = 0$. We denote the constant terms in the Laurent expansion at $\rho = 0$ by $\alpha(x_1, y)$ and $\beta(x_1, y)$ respectively. 
	Thus, as currents on $\bbD$, we have
	\begin{align}
	\lie g^o(x_1,y) 
	&= \varphi_{\KM}^o(x_1) \wedge \lie g^o(y) + d \alpha(x_1,y) + d^c \beta(x_1,y)  \\
	& \qquad \qquad   +  CT_{\rho = 0}   \int_1^{\infty} \nu^o(\sqrt u x_1) \wedge \delta_{\bbD^+_y} \wedge \Omega^{n-1-r(y)}  \ \frac{du}{u^{1+\rho}}  \label{eqn:go xy decomp}
	\end{align} 
	for all $x_1 \in V_1$ and $y \in (V_1)^{n-1}$.
	
	As a final observation, note that if $x_1 \in \mathrm{span}(y)$, then $\nu^o(\sqrt u x_1) \wedge \delta_{\bbD^+_y} =\delta_{\bbD^+_y}$, see \cite[Lemma 2.4]{GarciaSankaran}. Thus
	\begin{align}
			\gamma(x_1, y) &:=	CT_{\rho = 0}   \int_1^{\infty} \nu^o(\sqrt u x_1) \wedge \delta_{\bbD^+_y} \wedge \Omega^{n-1-r(y)}  \ \frac{du}{u^{1+\rho}} \\
			& = \begin{cases} 
				\lie g^o(x_1) \wedge \delta_{\bbD^+_y} \wedge \Omega^{n-1-r(y)} , & \text{if } x_1 \notin \mathrm{span}(y) \\ 
				0 , & \text{if } x_1 \in \mathrm{span}(y).
				\end{cases}
	\end{align}
	In the case that the components of $x = (x_1, y) = (x_1, \dots , x_n)$ are linearly independent, we recover the star product formula from \cite[Theorem 2.16]{GarciaSankaran}. 
	
	Now we discuss a decomposition of the global Green current $\lie g^o(T, \bfv)$, for $\bfv \in \Sym_{n}(F_{\bbR})_{\gg 0}$. Write 
	\begin{equation}
		v := \sigma_1(\bfv) = \begin{pmatrix} v_1  & v_{12} \\ v_{12}' & v_2 \end{pmatrix}
	\end{equation} with $v_1 \in \bbR_{>0}$ and $v_{12} \in M_{1,n-1}(\bbR)$; recall that $\sigma_1 \colon F \to \bbR$ is the distinguished real embedding. Set 	
		\begin{equation}
			v_2^* := v_2 - v'_{12}v_{12} / v_1  \in \Sym_{n-1}(\bbR)_{>0}, 
		\end{equation}
		and fix a matrix $a_2^* \in \GL_{n-1}(\bbR)$ such that $v_2^* = a_2^* \cdot (a_2^*)'$.
		
	\begin{proposition} \label{gT decomp}
	Let $T \in \Sym_n(F)$ and $\bfv \in \Sym_n(F_{\bbR})_{\gg 0}$ as above, and define $S(L)^{\vee}$-valued currents $\lie a(T,\bfv)$ and $\lie b(T,\bfv)$ on $X$ by the formulas
		\begin{equation} \label{eqn:lie a def}
			\lie a(T,\bfv)(\varphi) := \sum_{\bfx \in \Omega(T) } \varphi(\bfx) \gamma(\sqrt{v_1} x_1, y),
		\end{equation}
		where we have written $\sigma_1(\bfx) = (x_1, y) \in V_1 \oplus (V_1)^{n-1}$, and 
			\begin{equation}
				\lie b(T,\bfv)(\varphi) = \sum_{\bfx \in \Omega(T) } \varphi(\bfx) \,  \varphi^o_{\KM}  \left(\sqrt{v_1} x_1 +  \frac{y \cdot  v_{12}'}{\sqrt{ v_1} } \right)  \wedge \lie g^o(y a_2^*).
			\end{equation}
		Then  
			\begin{equation}
				\lie g^o(T,\bfv) (\varphi)   \equiv \lie a(T,\bfv)(\varphi) + \lie b(T,\bfv)(\varphi)  \pmod{\mathrm{im} \,\partial \, + \, \mathrm{im} \, \overline{\partial}}.
			\end{equation}

		\begin{proof}
			First, the fact that the sums defining $\lie a(T,\bfv)$ and $\lie b(T,\bfv)$ converge to currents on $X$ follows from the same argument as \cite[Proposition 4.3]{GarciaSankaran}.
			
			Now recall that 
				\begin{equation}
					\lie g^o(T,\bfv) (\varphi) = \sum_{\bfx \in \Omega(T) } \varphi(\bfx) \, \lie g^o( xa)
				\end{equation}             
		 	where $x = \sigma_1(\bfx)$, and $a \in \GL_n(\bbR)$ is any matrix satisfying $v = aa'$. Note that   			
				\begin{equation}
					v = \begin{pmatrix} v_1  & v_{12} \\ v_{12}' & v_2 \end{pmatrix} = \theta \begin{pmatrix} v_1 & \\ & v_2 ^*  \end{pmatrix}  \theta',  
					\qquad \text{ where } \theta = \begin{pmatrix} 1 & \\   v_{12}' / v_1 & 1_{n-1} \end{pmatrix} .
				\end{equation}
			Thus, we may take 
				\begin{equation} \label{eqn:a theta}
						a = \theta \cdot \psm{ \sqrt{v_1} & \\ & a_2^*} ,
				\end{equation}
			and so, applying \eqref{eqn:go xy decomp} , we find
				\begin{align}
					\lie g^o(T,\bfv)(\varphi)  	
					&=   \sum_{\bfx \in \Omega(T)} \, \varphi(\bfx) \, \lie g^o\left( (x_1, y) \theta  \psm{ \sqrt{v_1} 		& \\ & a_2^*} \right)  \qquad \qquad  x = (x_1, y)\\
					&=  \sum_{\bfx \in \Omega(T)}  \varphi(\bfx) \, \lie g^o \left( \sqrt{v_1} x_1 +  \frac{y \cdot  v_{12}'}{\sqrt{ v_1} } , \, y a_2^* \right)    \\
					&= \sum_{\bfx \in \Omega(T)} \varphi(\bfx) \Big( \varphi^o_{\KM} \left( \sqrt{v_1} x_1 +  \frac{y \cdot  v_{12}'}{\sqrt{ v_1} } \right) \wedge \lie g^o(y a_2^*) + \partial \alpha( \sqrt{v_1} x_1 +  \frac{y \cdot  v_{12}'}{\sqrt{ v_1} } ,y a_2^*) \notag \\
					& \qquad \qquad\qquad\qquad\qquad \ + \overline \partial \beta( \sqrt{v_1} x_1 +  \frac{y \cdot  v_{12}'}{\sqrt{ v_1} } , y a_2^*) + \gamma( \sqrt{v_1} x_1 +  \frac{y \cdot  v_{12}'}{\sqrt{ v_1} } ,y a_2^*) \Big). \label{eqn:go(T) decomp big sum}
			\end{align}
			
			Again, an argument as in  \cite[Proposition 4.3]{GarciaSankaran} shows that the sums
			\begin{equation}
			\eta_1 :=  \sum_{\bfx \in \Omega(T)} \varphi(\bfx) \,  \alpha( \sqrt{v_1} x_1 +  \frac{y \cdot  v_{12}'}{\sqrt{ v_1} } ,y a_2^*)
			\end{equation}
			and
			\begin{equation}
			\eta_2 := \sum_{\bfx \in \Omega(T)} \varphi(\bfx) \,  \beta( \sqrt{v_1} x_1 +  \frac{y \cdot  v_{12}'}{\sqrt{ v_1} } ,y a_2^*)
			\end{equation}
			converge to $\Gamma$-invariant currents on $\bbD$, and hence define currents on $X$. Moreover, it follows easily from the definitions that 
			\begin{equation}
			\gamma( \sqrt{v_1} x_1 +  \frac{y \cdot  v_{12}'}{\sqrt{ v_1} } ,y a_2^*) = \gamma(\sqrt{v_1} x_1, y).
			\end{equation}
			Thus, we find 
				\begin{equation}
					\lie g^o(T,\bfv)(\varphi)= \lie a(T,\bfv)(\varphi)+ \lie b(T,\bfv)(\varphi) + \partial \eta_1 + \overline{\partial } \eta_2,
				\end{equation} 
			as required. 
		\end{proof}
	\end{proposition}    
	Next, we define an $S(L^n)^{\vee}$-valued current $\psi(T,\bfv)$ as follows. 
For $\bfx \in \Omega(T)$, write $\sigma_1(\bfx) = x = (x_1, y) \in V_1 \oplus V_1^{n-1}$ as above; then
	\begin{align} \label{eqn:psi current def}
	\psi(T,\bfv)(\varphi)  \ &:= \ \sum_{\bfx \in \Omega(T)} \varphi(\bfx) \,		\varphi^o_{\KM}( \sqrt{v_1} x_1) \wedge \delta_{\bbD^+_y}  \wedge \Omega^{n-1-r(y)}
	\end{align}
	defines a $\Gamma$-equivariant current on $\mathbb D^+$, and hence descends to a current (also denoted $\psi(T,\bfv)$) on $X(\bbC)$.

	\begin{lemma}  \label{a and b ddc eqn}
		
		\begin{enumerate}[(i)]
		\item	Let   $\omega(T,\bfv)$ be the $T$th coefficient of the Kudla-Millson theta function, as in 	\eqref{eqn:KM theta q coeff}; then
			\begin{equation}
				\ddc \lie b(T,\bfv) \ = \ \omega(T,\bfv) - \psi(T,\bfv).
			\end{equation}
		\item  We have
		\begin{equation}
		\ddc \, \lie a(T,\bfv) + \delta_{Z(T)(\bbC)}\wedge \Omega^{n - r(T)} \ = \ \psi(T,\bfv),
		\end{equation}
		where $r(T) = \mathrm{rank}(T)$.  
	\end{enumerate}
		\begin{proof}
With $v = \sigma_1(\bfv)$ and taking $a = \theta \cdot \psm{ \sqrt{v_1} & \\ & a_2^*}$ as \eqref{eqn:a theta}, we have
\begin{align}
\omega(T,\bfv)(\varphi) \ 	&= \ \sum_{\bfx \in \Omega(T)} \varphi(\bfx) \,  \varphi^o_{\KM}(xa)  \\
&=  \ \sum_{\bfx \in \Omega(T)} \varphi(\bfx) \, \varphi^o_{\KM} \left( \sqrt{v_1} x_1 +  \frac{y \,  v_{12}'}{\sqrt{ v_1} } , \ y a_2^* \right) \\
&= \sum_{\bfx \in \Omega(T)} \varphi(\bfx) \, \varphi^o_{\KM} \left( \sqrt{v_1} x_1 +  \frac{y   v_{12}'}{\sqrt{ v_1} } \right) \wedge \varphi^o_{\KM} \left( y a_2^* \right)
\end{align}
for $\varphi \in S(L^n)$, where the last line follows from \cite[Theorem 5.2(i)]{KudlaMillsonIHES}.	Therefore, 
\begin{align}
\ddc \lie b(T, \bfv)(\varphi)  &=   \sum_{\bfx \in \Omega(T) } \varphi(\bfx) \,  \varphi^o_{\KM}  \left(\sqrt{v_1} x_1 +  \frac{y \,  v_{12}'}{\sqrt{ v_1} } \right)  \wedge  \ddc \lie g^o(y a_2^*) \\
&=  \sum_{\bfx \in \Omega(T) } \varphi(\bfx) \,  \varphi^o_{\KM}  \left(\sqrt{v_1} x_1 +  \frac{y \,  v_{12}'}{\sqrt{ v_1} } \right)  \wedge \Big\{  - \delta_{\bbD^+_y} \wedge \Omega^{n-1-r(y)}     + \varphi^o_{\KM}(ya_2^*) \Big\} \\
&=  - \sum_{\bfx \in \Omega(T) } \varphi(\bfx) \,  \varphi^o_{\KM}  \left(\sqrt{v_1} x_1 +  \frac{y \,  v_{12}'}{\sqrt{ v_1} } \right)  \wedge  \delta_{\bbD^+_y} \wedge \Omega^{n-1-r(y)}   + \omega(T,\bfv)(\varphi).
\end{align}
For $v \in V_1 $, the restriction  $\varphi^o_{\KM}(v) \wedge \delta_{\bbD^+_y}$ depends only on the orthogonal projection of $v$ onto $\mathrm{span}(y)^{\perp}$; see, for example,  \cite[Lemma 2.4]{GarciaSankaran}. In particular, 
	\begin{equation}
		\varphi^o_{\KM}  \left(\sqrt{v_1} x_1 +  \frac{y \,  v_{12}'}{\sqrt{ v_1} } \right)  \wedge  \delta_{\bbD^+_y}  = \varphi^o_{\KM}  \left(\sqrt{v_1} x_1   \right)  \wedge  \delta_{\bbD^+_y} .
	\end{equation}
	The first part of the lemma follows upon applying the definition of $\gamma(T,\bfv)$ in \eqref{eqn:psi current def}.
	
	The second part then follows from the first, together with \Cref{gT decomp} and \eqref{eqn:global green eqn}. 
		\end{proof}
	\end{lemma}

We finally arrived at the promised decomposition of $\widehat Z(T,\bfv)$. Recall that in defining the cycle $\widehat Z(T,\bfv)$ in \Cref{sec:prelim arith cycles}, we fixed a representative $(Z_0, g_0)$ for $\widehat \omega^{n- r(T)}$ such that $Z_0$ intersects $Z(T)$ properly. By the previous proposition,
\begin{equation}
	\ddc \left( \lie a(T,\bfv) +  g_0 \wedge \delta_{Z(T)(\bbC) } \right)  + \delta_{Z(T)\cap Z_0 (\bbC)} \ = \ \psi(T,\bfv);
\end{equation} we then obtain classes in $\ChowHat{n}_{\bbC}(X, \Dcur)\otimes_{\bbC}S(L^n)$ by setting
\begin{equation}
	\widehat A(T,\bfv) \ := \ \left( Z(T) \cdot Z_0, \, [ \psi(T,\bfv),  \, \lie a(T,\bfv) + g_0 \wedge  \delta_{Z(T)(\bbC)}   ] \right)  
\end{equation}
and
\begin{equation}
\widehat B(T,\bfv) \ := \ \left(0, \, [ \omega(T,\bfv) - \psi(T,\bfv) , \, \lie b(T,\bfv)] \right),
\end{equation}
so that
	\begin{equation}
		\widehat Z(T,\bfv) = \widehat A(T,\bfv) + \widehat B(T,\bfv) \ \in \ \ChowHat{n}_{\bbC}(X,\Dcur) \otimes_{\bbC} S(L)^{\vee}
	\end{equation}
	
	\begin{remark} \label{remark: A vanishes if not semi definite}
		Suppose $T = \psm{* & * \\ * &T_2}$ as above; if $T_2$ is not totally positive semidefinite, then $\bbD^+_y = \emptyset$ for any $\bfy \in \Omega(T_2)$, and hence $\widehat A(T,\bfv) = 0$. 
	\end{remark}

	\section{Modularity I}
	In this section, we establish the modularity of the generating series 
	\begin{equation}
		\widehat \phi_B (\bftau) = \sum_{T = \psm{* & * \\ * & T_2}} \, \widehat B(T,\bfv) \, q^T.	\end{equation}
	Note that 
	\begin{equation}
		\widehat B(T,\bfv) = (0, [\ddc \lie b(T,\bfv ), \lie b(T,\bfv)])  = a(\lie b(T,v));
	\end{equation}
	thus in light of \Cref{definition of modularity}, it suffices to establish the following theorem.
	\begin{theorem} \label{thm:B series}
	Fix $T_2 \in \Sym_{n-1}(F)$, and  consider the generating series
			\begin{equation} \label{eqn:B's gen series}
				 \xi(\bftau) = \sum_{T = \psm{* & * \\ * & T_2}} \lie b(T,\bfv) \, q^T ,  
			\end{equation}
Then $ \xi(\bftau)$ is an element of $A_{\kappa, T_2}(\rho_L^{\vee}; D^*(X))$, see \Cref{def:smooth currents}.

\begin{proof} We begin by showing the convergence of the series \eqref{eqn:B's gen series}. By definition,
	\begin{equation}
 		\sum_{T = \psm{T_1 &T_{12}\\ T_{12}' & T_2}} \, \lie b(T, \bfv)(\varphi) \,  q^T  
	=   \sum_{T} \sum_{(\bfx_1, \bfy) \in \Omega(T)}    \varphi(\bfx_1,\bfy) \,  \varphi_{\KM}^o \left( \sqrt{v_1} x_1+ \frac{y \cdot v'_{12} }{ \sqrt{v_1}} \right)\wedge \lie g^o(ya_2^*)  \ q^T \label{eqn:expansion of b sum}
	\end{equation}
	where $ x_1 = \sigma_1(\bfx_1)$ and $y = \sigma_1(\bfy)$, and here we are working with $\Gamma$-equivariant currents on $\bbD^+$.  For $v \in V_1$, consider the normalized Kudla-Millson form
	\begin{equation}
		\varphi_{\KM}(v) \ := \ e^{- 2 \pi \langle v, v \rangle} \, \varphi^o_{\KM}(v)
	\end{equation}
	which is a Schwartz form on $V_1$, valued in closed forms on $\bbD^+$;  more precisely, for any integer $k$ and relatively compact open subset $U \subset \bbD^+$, there exists a positive definite quadratic form $Q_U$ on $V_1$ such that
	\begin{equation} \label{eqn:varphi estimate}
			\| \varphi_{\KM}(v) \|_{k,\overline U} \ll  e^{ - Q_U(v)}
	\end{equation}
	where $\| \cdot \|_{k,\overline U}$ is an algebra seminorm measuring uniform convergence of all derivatives of order $\leq k$ on the space of smooth forms supported on $\overline U$, and the implied constant  depends on $k$ and $\overline U$.  
	Similarly, for $y \in V_1^{n-1}$, write 
	\begin{equation}
		\lie g (y) = e^{-2 \pi  \sum \langle y_i, y_i \rangle } \lie g^o(y); 
	\end{equation}
	if $\bbD^+_y \cap \overline U = \emptyset$, then $\lie g(y)$ is smooth on $U$, and the form $Q_U$ may be chosen so that 
	\begin{equation} \label{eqn:g(y) estimate}
		\| \lie g(y) \|_{k,\overline U}  \ \ll \ e^{- \sum_{i=1}^{n-1}  Q_U(y_i)} , \qquad y = (y_1, \dots, y_{n-1}) \in V_1^{n-1},
	\end{equation}
	see \cite[\S 2.1.5]{GarciaSankaran}. 
	
	Finally, for the remaining real embeddings $\sigma_2, \dots \sigma_d$, let $\varphi_{\infty_i} \in S(V_i^n)$ denote the standard Gaussian on the positive definite space $V_i = V \otimes_{F, \sigma_i} \bbR$, defined by $\varphi_{\infty_i}(x_1, \dots, x_n) = e^{ - 2 \pi \sum \langle x_i, x_i \rangle}$. 
	Then a brief calculation gives
	\begin{equation}
		\xi(\bftau)(\varphi) = \sum_{T} \sum_{(\bfx_1, \bfy) \in \Omega(T)} \varphi(\bfx_1,\bfy) \varphi_{\KM} \left( \sqrt{v_1} x_1+ \frac{y \cdot v'_{12} }{ \sqrt{v_1}} \right)\wedge \lie g(ya_2^*)   \cdot \prod_{i=2}^d \varphi_{\infty_i}( \sigma_i (\bfx_1, \bfy) a_i)  \ e( T \mathbf u )
	\end{equation}
	where we have chosen  matrices $a_i \in GL_n(\bbR)$ for $i = 2, \dots, d$, such that $\sigma_i(\bfv) = a_i \cdot a_i'$.  
	Let 
		\begin{equation}
			S_1 :=  \{ \bfy \in (L')^{n-1} \ | \langle \bfy, \bfy \rangle = T_{2}  \text{ and } \bbD^+_{y} \cap \overline U \neq \emptyset  \}
		\end{equation}
	which is a finite set, and let
		\begin{equation}
			S_2 :=  \{ \bfy \in (L')^{n-1} \ | \langle \bfy, \bfy \rangle = T_{2} \text{ and } \bbD^+_{y} \cap \overline U = \emptyset  \}.
		\end{equation}
		
	Using the estimates \eqref{eqn:varphi estimate} and \eqref{eqn:g(y) estimate}, and standard arguments for convergence of theta series, it follows that  the sum
		\begin{equation} \label{eqn:xi_0 sum}
		\sum_{T = \psm{* & * \\ * & T_2}} \sum_{\substack{(\bfx_1, \bfy ) \in \Omega(T) \\ \bfy \in S_2}}\varphi(\bfx_1,\bfy) \varphi_{\KM} \left( \sqrt{v_1} x_1+ \frac{y \cdot v'_{12} }{ \sqrt{v_1}} \right)\wedge \lie g(ya_2^*)    \prod_{i=2}^d \varphi_{\infty_i}( \sigma_i (\bfx_1, \bfy) a_i)  \ e( T \mathbf u )
		\end{equation}
		converges absolutely to a smooth form on $\bbH^d_n \times U$.  The (finitely many) remaining terms, corresponding to $\bfy \in S_1$, can be written as 
	\begin{equation}
		\sum_{\bfy \in S_1}   f_{\bfy}(\bftau)(\varphi)\wedge \lie g(y a_2^*)
	\end{equation}
	where, for any $\bfy \in V^{n-1}$ and $\varphi \in S(L^n)$, we set
	\begin{equation}
		f_{\bfy}(\bftau) (\varphi) \ = \  \sum_{ \bfx_1 \in V}  \varphi(\bfx_1, \bfy) \varphi_{\KM} \left( \sqrt{v_1} x_1+ \frac{y \cdot v'_{12} }{ \sqrt{v_1}} \right)   \prod_{i=2}^d \varphi_{\infty_i}( \sigma_i (\bfx_1, \bfy)a_i)  \ e( T(\bfx_1, \bfy)\mathbf u ),
	\end{equation}
	where $T(\bfx_1, \bfy) = \psm{ \langle \bfx_1, \bfx_1 \rangle & \langle \bfx_1, \bfy \rangle \\ \langle \bfx_1, \bfy \rangle' & \langle \bfy, \bfy \rangle } $.
	Again, the estimate \eqref{eqn:varphi estimate} shows that the series defining $f_{\bfy}(\bftau)$ converges absolutely to a smooth form on $\bbH_d^n \times \bbD^+$. Moreover, for a fixed $y \in V_1^{n-1}$ and any compactly supported test form $\alpha$ on $\bbD^+$, the value of the current $\lie g^o(y a_2^*)[\alpha]$  varies smoothly in the entries of $a_2^*$ (this fact follows easily from the discussion in \cite[\S 2.1.4]{GarciaSankaran}). 
	
	Taken together, the above considerations imply that the series $\xi(\bftau)(\varphi)$ converges absolutely to a $\Gamma$-invariant current on $\bbD^+$, and therefore descends to a current on $X$ that satisfies part (b) of \Cref{def:smooth currents} as $\bftau$ varies. In addition, this discussion shows that given any test form $\alpha$, the value of the current $\xi(\bftau)[\alpha]$ is smooth in $\bftau$. 
	
	It remains to show that $\xi(\bftau)$ transforms like a Jacobi modular form, i.e.\ is invariant under the slash operators \eqref{eqn:slash operators}. 
	Recall that   the form $\varphi_{\KM}$ is of weight $p/2 +1$; more precisely, let $\widetilde U(1) \subset \Sptilde_1(\bbR)$ denote the inverse image  of $U(1)$, which admits a genuine character $\chi$ whose square is the identity on $U(1)$. Then $\omega(\widetilde k) \varphi_{\KM} = (\chi(\widetilde k))^{p+2} \varphi_{\KM}$ for all $\widetilde k \in \widetilde U(1)$, where $\omega$ is the Weil representation attached to $V_1$, cf.\ \cite[Theorem 5.2]{KudlaMillsonIHES}.
	
	To show that $\xi(\bftau)$ transforms like a Jacobi form, note that  (by Vaserstein's theorem), every element of $\widetilde \Gamma^J$ can be written as a product of the following elements. 
	
	\begin{enumerate}[(i)]
		\item For each $i = 1, \dots, d$, let 
		\begin{equation} \label{eqn:Jacobi generator epsilon}
		\widetilde \epsilon(i) = (\widetilde \epsilon (i))_v \in  \prod_{v| \infty} \Sptilde_n(F_v)
		\end{equation} be the element whose $v$'th component is $(\Id, 1)$ if $v \neq \sigma_i$, and $(\Id, -1)$ if $v= \sigma_i$. 
		\item For $\mu, \lambda \in \mathrm{M}_{1, n-1}(\calO_F)$, 
		let 
		\begin{equation} \label{eqn:Jacobi generator gamma_lambda mu}
	\gamma_{\lambda, \mu} =\left( \begin{array}{cc|cc} 
											1 & 0 & 0 & \mu    \\
											\lambda' & 1_{n-1} & \mu' & 0 \\
											\hline
											0 & 0 & 1 & - \lambda \\ 
											0&0&0&1_{n-1}
											\end{array}  \right)  \in G^J(\calO_F).
		\end{equation}
		Let $\iota_F(\gamma_{\lambda,\mu}) \in \Sptilde_{n,\bbA}$ denote its image under the splitting \eqref{eqn:metaplectic canonical splitting}; we choose $\widetilde \gamma_{\lambda, \mu} \in \widetilde \Gamma^J$ to be the archimedean part of a representative $\iota_F(\gamma_{\lambda, \mu}) = \widetilde \gamma_{\lambda,\mu} \cdot \widetilde \gamma_f$.
		
		\item For $r \in \calO_F$, let 
		\begin{equation} \label{eqn:Jacobi generator gamma_r}
				\gamma_r = 	\left( \begin{array}{cc|cc} 
											1 & 0 & r & 0   \\
											0 & 1_{n-1} & 0 & 0 \\
											\hline
											0 & 0 & 1 & 0 \\ 
											0&0&0&1_{n-1}
											\end{array}  \right) \in G^J(\calO_F),
		\end{equation}
		and choose an element $\widetilde \gamma_r$  as the archimedean part of a representative of  $\iota_F(\gamma_r)$, as before.
		
		\item Finally, let 	
			\begin{equation} \label{eqn:Jacobi generator S}
			S = \left( \begin{array}{cc|cc} 
										0 & 0 & -1 & 0   \\
										0 & 1_{n-1} & 0 & 0 \\
										\hline
										1 & 0 & 0 & 0 \\ 
										0&0&0&1_{n-1}
										\end{array}  \right)
			\end{equation}
			and take $\widetilde S \in  \widetilde \Gamma^J$ to be the archimedean part of a representative of $\iota_F(S)$. 
	\end{enumerate}

			Now, rearranging the absolutely convergent sum \eqref{eqn:xi_0 sum}, we may write
			\begin{equation}
			\xi(\bftau) \ = \ \sum_{\bfy \in \Omega(T_2)} f_{\bfy}(\bftau) \wedge \lie g(y a_2^*).
			\end{equation}	
				Using the aforementioned generators, a direct computation shows that $\bfv_2^*  = \bfv_2 - \bfv'_{12} \bfv_{12} / \bfv_1$, viewed as a function on $\bbH^d_n$, is invariant under the action of $\widetilde \Gamma^J$; it therefore suffices to show that for a fixed $\bfy$, the $S(L^n)^{\vee}$-valued function $f_{\bfy}(\bftau)$ transforms like a Jacobi form. 
			
			It is a straightforward verification to check that $ f_{\bfy}(\bftau)$ is invariant under the action of $\widetilde \epsilon(i)$,   $\tilde \gamma_{\lambda, \mu}$, and $\tilde \gamma_r$. For example, the element $\widetilde \gamma_{\lambda, \mu}^{-1}$ acts on $S(L^n)$ by the formula
			\begin{equation}
				\rho(\widetilde \gamma_{\lambda, \mu}^{-1}) (\varphi) \left( \bfx_1, \bfy \right) \ = \ e \big( 2  \langle \bfx_1,\bfy \rangle \mu'  - \langle \bfy \lambda' , \bfy ' \rangle \mu'\big) \, \varphi( \bfx_1 -\bfy  \lambda', \, \bfy)
			\end{equation}
			and $\gamma_{\lambda, \mu} $ acts on $ \bbH_n^d$ by the formula
			\begin{equation}
				\gamma_{\lambda, \mu} \cdot \bftau =  \begin{pmatrix} \bftau_1 & \bftau_{12} + \bftau_1 \lambda + \mu \\ \bftau_{12}' + \bftau_1 \lambda' + \mu'  & \bftau_2 + \left( \lambda' \cdot \bftau_{12} + \bftau_{12}' \cdot \lambda \right) + \mu' \cdot \lambda \end{pmatrix} 
			\end{equation}
			where $\bftau= \psm{\bftau_1 & \bftau_{12} \\ \bftau_{12}' & \bftau_2}$. Moreover, writing $\widetilde \gamma_{\lambda, \mu} = (\gamma_{\lambda, \mu}, (\phi_v))_v$ as in \Cref{sec:modularity def}, we have $\prod \phi_v(\tau) = 1$.  For $\bfx_1 \in V$ and $\bfy \in V^{n-1}$, a direct computation gives
			\begin{equation}
			 	\mathrm{tr} \Big( T(\bfx_1, \bfy) \cdot \mathrm{Re}(\gamma_{\lambda, \mu} \cdot \bftau 
			 )\Big) =				  \mathrm{tr} \Big( T(\bfx_1 + \bfy \lambda', \bfy) \bfu \Big)	+   2  \langle \bfx_1, \bfy \rangle \mu'  +  \langle \bfy \lambda', \bfy  \rangle\mu';
			\end{equation}
			 therefore, applying the above identity and the change of variables $\bfx_1 \mapsto \bfx_1 - \bfy \cdot \lambda'$, we find
			\begin{align}
			f_{\bfy}(\gamma_{\lambda, \mu}\cdot \bftau)(\varphi) &= \sum_{\bfx_1 \in V} \varphi(\bfx_1, \bfy)	  \varphi_{\KM} \left( \sqrt{v_1}( x_1 + y \cdot \lambda')+ \frac{y \cdot v'_{12} }{ \sqrt{v_1}} \right)  \notag \\
			  & \qquad \qquad \times   \left\{\prod_{i=2}^d \varphi_{\infty_i} \left( \sigma_i (\bfx_1, \bfy)\psm{1 & \\ \lambda' & 1 } a_i  \right)  \right\} \, e \Big(T(\bfx_1, \bfy) \mathrm{Re}(\gamma_{\lambda, \mu} \bftau) \Big)
			    \\
			  &= \sum_{\bfx_1 \in V}  \left\{   \varphi(\bfx_1 - \bfy \bflambda') \, e(2 \langle \bfx_1, \bfy \rangle \mu ' - \langle \bfy \lambda ', \bfy \rangle \mu') \right\} \varphi_{\KM} \left( \sqrt{v_1} x_1+ \frac{y \cdot v'_{12} }{ \sqrt{v_1}} \right) \notag  \\
			  & \qquad \qquad \times \prod_{i=2}^d \varphi_{\infty_i}( \sigma_i (\bfx_1, \bfy)a_i)  \ e( T(\bfx_1, \bfy)\mathbf u ) \\ 
			  &= f_{\bfy}(\bftau) \left( \rho( \widetilde \gamma_{\lambda, \mu}) \varphi \right)
			\end{align}
		as required.
			
			As for $\widetilde S$, recall that  $\iota_F(S)$  acts on $S(V(\bbA)^n)$ by the partial Fourier transform in the first variable;  the desired invariance follows from Poisson summation on $\bfx_1$ and straightforward identities for the behaviour of the Fourier transform under translations and dilations.
	\end{proof}
	
	\end{theorem}

	\section{Modularity II}
	In this section, we prove the modularity of the generating series $\widehat \phi_A(\bftau)$. 	By \Cref{remark: A vanishes if not semi definite}, we only need to consider totally positive semidefinite matrices $T_2$; assume that this is case throughout this section.

	We begin by fixing an element $\bfy = (\bfy_1, \dots, \bfy_{n-1}) \in \Omega(T_2)$, and setting $y = \sigma_1(\bfy)$. Let 
	\begin{equation}
		U_{\bfy} = \mathrm{span}(\bfy_1, \dots, \bfy_{n-1}) \subset V, 
	\end{equation}
	so that $U_{\bfy}$ is totally positive definite. Let 
	\begin{equation}
		\Lambda_{\bfy} := U_{\bfy} \cap L, \qquad \text{and} \qquad  \Lambda_{\bfy}^{\perp} := U_{\bfy}^{\perp} \cap L
	\end{equation}
	and set
	\begin{equation}
		\Lambda := \Lambda_{\bfy} \oplus \Lambda_{\bfy}^{\perp}  \subset L,
	\end{equation}
	so that 
	\begin{equation}
		\Lambda \subset L \subset L' \subset \Lambda'.
	\end{equation}
	In light of the definition \eqref{eqn:intro S(L) defn}, we have a natural inclusion $S(L^n) \to S(\Lambda^n)$, and the composition
	\begin{equation}  \label{eqn:lattice embedding}
	 S(L^n) \to S(\Lambda^n) \stackrel{\sim}{\to} S(\Lambda_{\bfy}^n) \otimes S((\Lambda_{\bfy}^{\perp})^{n}).
	\end{equation}
is equivariant for the action of $\widetilde \Gamma^J$, via $\rho_L$ on the left hand side, and via $\rho_{\Lambda_{\bfy}} \otimes \rho_{\Lambda_{\bfy}^{\perp}}$ on the right; this latter fact can be deduced from  explicit formulas for the Weil representation, cf.\ \cite[Proposition II.4.3]{KudlaCastle}.

  Note that $U_{\bfy}^{\perp}$ is a quadratic space of signature $((p',2 ), (p'+2, 0 ), \dots (p'+2, 0)) $ with $p'  = p- \mathrm{rank}(T_2)$, so the constructions in \Cref{sec:Prelims} apply equally well in this case. In particular, let $X_{\bfy}(\bbC) = \Gamma_{\bfy} \big\backslash \bbD^+_{y}$. Then for  $m  \in F$ and $\bfv_1 \in (F \otimes_{\bbR} {\bbR})_{\gg0}$, we have a special divisor
	\begin{equation} \label{eqn:genus one Zy def}
		\widehat Z_{U_{\bfy}^{\perp}}(m, \bfv_1)  
\in \ChowHat{1}_{\bbC}(X_{\bfy}) \otimes S(\Lambda_{\bfy}^{\perp})^{\vee},
	\end{equation}
	where we introduce  the subscript $U_{\bfy}^{\perp}$ in the notation to emphasize the underlying quadratic space being considered.
	
	Let 
	\begin{equation}
		\pi_{\bfy} \colon X_{\bfy} \to X
	\end{equation} 
	denote the natural embedding, whose image is the cycle $Z(\bfy)$ of codimension $\rank(T_2)$, and define a class 
		\begin{equation}
			\widehat Z_{\bfy}(m, \bfv_1) \in  \ChowHat{{\rk(T_2)+1}}(X, \Dcur) \otimes_{\bbC} S((\Lambda_{\bfy}^{\perp})^n) 
		\end{equation}
		as follows: suppose $\varphi \in S((\Lambda_{\bfy}^{\perp})^n) $ is of the form $\varphi_1 \otimes \varphi_2$ with $\varphi_1 \in S(\Lambda_{\bfy}^{\perp}) $ and $\varphi_2 \in S((\Lambda_{\bfy}^{\perp})^{n-1})$.
	 Then, using the pushforward $\pi_{\bfy,*}$,  set 
	\begin{equation}
	\widehat Z_{\bfy}(m, \bfv_1)(\varphi_1 \otimes \varphi_2) := \varphi_2(0) \cdot	\pi_{\bfy,*}  \left(   \widehat Z_{U_{\bfy}^{\perp}}(m, \bfv_1,\varphi_1)\right) \ \in \ \ChowHat{\rk(T_2) + 1}(X, \Dcur),
	\end{equation}
	and extend this definition to arbitrary $\varphi$ by linearity. Observe that the pushforward is an element of  $\ChowHat{n}_{\bbC}(X, \Dcur)$;  the existence of   pushforward maps along arbitrary proper morphisms, which are not available in general for the Gillet-Soul\'e Chow groups, are an essential feature of the extended version, \cite[\S 6.2]{BurgosKramerKuhn}.
	
	Finally, for $\bftau = \psm{ \bftau_1 & \bftau_{12} \\ \bftau_{12}' & \bftau_2}  \in \bbH_d^n$, we define the generating series
	\begin{equation}
		\widehat \phi_{\bfy}(\bftau_1)  \ := \ \sum_{m \in F} \widehat Z_{\bfy}(m, \bfv_1) \,  q_1^m
	\end{equation}
	where $\bftau_1 \in \bbH_1^d$ with $\bfv_1 = \mathrm{Im}(\bftau_1)$, and $q_1^m = e(m \bftau_1)$.

	There is  also  a classical  theta function attached to the totally positive definite space $U_{\bfy}$, defined as follows: let $\varphi \in S(\Lambda_{\bfy}^n)$ and suppose $\varphi = \varphi_1 \otimes \varphi_2$ with $\varphi_1 \in S(\Lambda_{\bfy})$ and $\varphi_2 \in S(\Lambda_{\bfy}^{n-1})$. Then we set
	\begin{equation}
					\theta_{\bfy} (\bftau)(\varphi_1 \otimes \varphi_2) \ := \  \varphi_2(\bfy) \sum_{\lambda \in U_{\bfy}} \varphi_1(\lambda) \ e\left( \langle \lambda, \lambda \rangle \bftau_1 + 2 \langle \lambda, \bfy \rangle \bftau_{12}'   \right) \, e(T_2\cdot \bftau_2),
	\end{equation}
	and again, extend to all $\varphi \in S(\Lambda_{\bfy}^n)$ by linearity.
 	It is well-known that $\theta_{\bfy}(\bftau)$ is a holomorphic Jacobi modular form of weight $\dim U_{\bfy} / 2  = \rk(T_2)/2$ and index $T_2$, see e.g. \cite[\S II.7]{EichlerZagier}.

 The Fourier expansion of $	\theta_{\bfy} (\bftau)(\varphi)$ can be written, for $\varphi = \varphi_1 \otimes \varphi_2$ as above, as
 	\begin{equation}
	 	\theta_{\bfy}(\bftau)(\varphi_1\otimes \varphi_2) \ =\  \varphi_2(\bfy)\sum_{T = \psm{* & * \\ * & T_2}}  r_{\bfy}(T, \varphi_1) \ q^T, 
 	\end{equation}
 	where $r_{\bfy}(T) \in S(\Lambda_{\bfy})^{\vee}$ is given by the formula
 	\begin{equation}
	 	r_{\bfy} \left( \psm{T_1 & T_{12} \\ T_{12}' & T_2 }, \varphi_1 \right)  = \sum_{\substack{\lambda \in U_{\bfy} \\ \langle \lambda, \lambda \rangle = T_1 \\ \langle \lambda, \bfy \rangle = T_{12}} } \, \varphi_1 (\lambda)
 	\end{equation}
 	
 	Finally,  note that given $T$ as above, we must have either $\rank(T) = \rank(T_2) + 1$, or $\rank(T) = \rank(T_2)$.
 	\begin{lemma} \label{r vanishing non-degenerate}
	 		 Suppose $\rank(T) = \rank(T_2) + 1$. Then for any $\bfy \in \Omega(T_2)$, we have $r_{\bfy}(T) = 0.$
	 		 
 		\begin{proof}
 			Suppose $r_{\bfy}(T) \neq 0$; then there exists a tuple $(\lambda, \bfy) \in \Omega(T)$  with $ \mathrm{span}(\lambda, \bfy) = \mathrm{span} (\bfy)$, which  contradicts the assumption on $\rank(T)$.

 		\end{proof}
 	\end{lemma}
	
	\begin{proposition}  \label{PhiA Prop}
		As formal generating series, we have
		\begin{equation} \label{eqn:PhiA Prop main eqn}
			\widehat{ \phi}_A(\bftau) = 	\sum_{T = \psm{ * & * \\ * & T_2}} \widehat A(T, \bfv) \, q^T = \sum_{\substack{ \bfy \in \Omega(T_2) \\ \mod{\Gamma}} }  \widehat{ \phi}_{\bfy}(\bftau)  \cdot \widehat{ \omega}^{n - r(T_2) - 1} \otimes \theta_{\bfy}(\bftau) ,
		\end{equation}
		where
		\begin{equation} 
			\widehat{ \phi}_{\bfy}(\bftau)  \cdot \widehat{ \omega}^{n - r(T_2) - 1} \ := \  \sum_{m \in F} \widehat Z_{\bfy}(m, \bfv_1) \cdot \widehat \omega^{n-r(T_2) - 1} \, q_1^m;
		\end{equation}
	here, we view the right hand side of \eqref{eqn:PhiA Prop main eqn} as valued in $S(L^n)^{\vee}$ by dualizing  \eqref{eqn:lattice embedding}.

		\begin{proof}
			By linearity, it suffices to evaluate both sides of the desired relation at a Schwartz function $\varphi \in S(L^n)$ of the form $\varphi = \varphi_1 \otimes \varphi_2$ for $\varphi_1 \in S(L) $ and $\varphi_2 \in S(L^{n-1})$.

			Then we may write
			\begin{align}
				Z(T)(\varphi_1 \otimes \varphi_2)  \ &= \ \sum_{\substack{\bfx \in \Omega(T) \\ \text{mod } \Gamma}}  (\varphi_1 \otimes \varphi_2) (\bfx)  \, Z(\bfx)  \\
				&= \sum_{\substack{\bfy \in \Omega(T_2)   \\ \text{mod } \Gamma}} \varphi_2(\bfy) \sum_{\substack{ \bfx_1 \in \Omega(T_1) \\ \langle \bfx_1, \bfy \rangle = T_{12} \\ \text{mod } \Gamma_{\bfy} }} \varphi_1(\bfx_1) \Gamma_{(\bfx_1, \bfy)} \big\backslash  \bbD_{(\bfx_1, \bfy)}.
			\end{align}
			We may further assume that 
			\begin{equation}
				\varphi_1 = \varphi_1' \otimes \varphi_1'' \in S(U_{\bfy}) \otimes S(U_{\bfy}^{\perp}) \qquad  \text{ and  } \qquad	\varphi_2  = \varphi_2' \otimes \varphi_2'' \in S(\Lambda_{\bfy}^{n-1}) \otimes S((\Lambda_{\bfy}^{\perp})^{n-1});
			\end{equation}
			in this case, $\varphi_2(\bfy) = \varphi_2'(\bfy) \varphi_2''(0)$. 
			
			For a vector $\bfx_1 \in V$ as above, write its orthogonal decomposition as
			\begin{equation}
				\bfx_1 = \bfx_1'  \ + \ \bfx_1''  \ \in \ U_{\bfy} \oplus U_{\bfy}^{\perp},
			\end{equation}
			and note that $\bbD^+_{(x_1, y)} = \bbD^+_{(x_1'', y)}$, where $x_1 = \sigma_1(\bfx_1)$, etc., and $\Gamma_{(\bfx_1, \bfy)} = \Gamma_{(\bfx_1'', \bfy)}$.

			Thus, decomposing the sum on $\bfx_1$ as above and writing $T =\psm{T_1 & T_{12} \\ T_{12}' & T_2}$, we have
			\begin{align} \notag
				Z(T)& (\varphi_1\otimes \varphi_2)  \\
				&=  \sum_{\substack{ \bfy \in \Omega(T_2) \\ \text{mod } {\Gamma}}} \varphi_2(\bfy) \, \sum_{m \in F} 
				\left(	
					\sum_{\substack{ 
								\bfx_1'' \in 	U_{\bfy}^{\perp} \\
								\langle \bfx_1'', \bfx_1'' \rangle = m \\
								\text{mod } \Gamma_{\bfy}
								}} 
					\varphi_1''(\bfx_1'')
					\Gamma_{(\bfx_1'', \bfy)} \big\backslash  	\bbD_{(\bfx_1'', \bfy)}.
				\right) 
				\cdot
				\left(	
					\sum_{\substack{
								\bfx_1' \in U_{\bfy} \\
								\langle \bfx_1' , \bfx_1' \rangle = T_1 - m \\
								\langle \bfx_1', \bfy \rangle = T_{12}
					}
				}
				\varphi'_1(\bfx_1')
				\right)  \notag \\
				&=  \sum_{\substack{ \bfy \in \Omega(T_2) \\ \text{mod } {\Gamma}}} \varphi_2(\bfy) \, \sum_{m \in F} 
				\left(	
				\sum_{\substack{ 
						\bfx_1'' \in 	U_{\bfy}^{\perp} \\
						\langle \bfx_1'', \bfx_1'' \rangle = m \\
						\text{mod } \Gamma_{\bfy}
				}} 
				\varphi_1''(\bfx_1'')
				\Gamma_{(\bfx_1'', \bfy)} \big\backslash  	\bbD_{(\bfx_1'', \bfy)}.
				\right) 
				\cdot
				r_{\bfy}\left( \psm{T_1 - m & T_{12} \\ T_{12'} & T_2}, \varphi_1' \right). 
			\end{align}
			which we may rewrite as
			\begin{align} \label{eqn:Z(T) cycle relation}
				Z(T)\left( \varphi_1 \otimes \varphi_2 \right) 
				&=
				\sum_{
					\substack{ \bfy \in \Omega(T_2)  \\ \text{mod } \Gamma}
				}
				\varphi_2''(0)\varphi_2'(\bfy)
				\sum_{m } 
				\pi_{\bfy, *} \left( Z_{U_{\bfy}^{\perp}} (m) (\varphi_1'') \right) \cdot  			
					r\left( \psm{T_1 - m & T_{12} \\ T_{12'} & T_2}, \varphi_1' \right)   \\
					&= 	\sum_{
						\substack{ \bfy \in \Omega(T_2)  \\ \text{mod } \Gamma}
					}
					\sum_{m } 
				Z_{\bfy} (m) (\varphi_1'' \otimes \varphi_2'')  \cdot  			
				 \left\{ \varphi_2'(\bfy) \, 	r\left( \psm{T_1 - m & T_{12} \\ T_{12'} & T_2}, \varphi_1' \right)  \right\}
			\end{align}
			where in the second line,  $Z_{\bfy}(m)$ denotes the $S((\Lambda_{\bfy}^{\perp})^n)^{\vee}$-valued cycle 
			\begin{equation}
			Z_{\bfy}(m) \colon \varphi''  \mapsto \varphi_2''(0) \,  \pi_{\bfy,*}   Z_{U_{\bfy}^{\perp} } (m,\varphi_1'').
			\end{equation}
			
			Now suppose that $\rk(T) = \rk(T_2)+1$. Then, by \Cref{r vanishing non-degenerate}, the term $m=0$ does not contribute to \eqref{eqn:Z(T) cycle relation}, and so  all the terms  $Z_{U_{\bfy}^{\perp}}(m)$ that do contribute are divisors. To incorporate Green currents in the discussion, note that, at the level of arithmetic Chow groups, the pushforward is given by the formula
			\begin{align}
			\widehat	Z_{\bfy}(m, \bfv_1)(\varphi'') &= 	\varphi_2''(0) \ \pi_{\bfy,*} \widehat Z_{U^{\perp}_{\bfy} } (m, \bfv_1, \varphi_1'') \\
			&=  \left( \pi_{\bfy, *}  Z_{U^{\perp}_{\bfy}}(m,\varphi_1''), \ \left[ \omega_{U_{\bfy}^{\perp}}(m, \bfv_1,\varphi_1'')\wedge \delta_{Z(\bfy)}, \lie g^o_{U_{\bfy}^{\perp}}(m,\bfv_1,\varphi_1'') \wedge \delta_{Z(\bfy)} \right] \right),
			\end{align}
 			where, as before, we use the subscript $U_{\bfy}^{\perp}$ to denote objects defined with respect to that space.
 			
 			This may be rewritten as
 			\begin{align}
				\widehat Z_{\bfy}(m, \bfv_1)(\varphi'') &= \widehat Z_{\bfy}(m)^{\can}(\varphi'')   \notag \\
				& \qquad + \varphi''_2(0)\left(0, \left[ \omega_{U^{\perp}_{\bfy}}(m, \bfv_1, \varphi''_1) \wedge \delta_{Z(\bfy)} - \delta_{Z_{\bfy}(m) } , \, \lie g^o_{U_{\bfy}^{\perp}}(m,\bfv_1,\varphi''_1) \wedge \delta_{Z(\bfy)} \right] \right) \label{eqn:Z(y) pushforward}
 			\end{align}
 			where $\widehat Z_{\bfy}(m)^{\can} = (Z_{\bfy}(m), [ \delta_{Z_{\bfy}(m) }, 0])$ is the canonical class associated to $Z_{\bfy}(m)$. 
 			Thus, 
 			\begin{equation}
 					\widehat Z_{\bfy}(m,\bfv_1) \cdot \widehat \omega ^{n-\rk(T)}   = \widehat Z_{\bfy}(m)^{\can} \cdot \widehat \omega^{n-\rk(T)}   + \left( 0, \, [\beta_{\bfy}(m, \bfv_1) , \alpha_{\bfy}(m, \bfv_1)]\right)
 			\end{equation}
 			where  $\alpha_{\bfy}(m,\bfv_1)$ and $\beta_{\bfy}(m,\bfv_1)$ are  $S((\Lambda_{\bfy}^{\perp})^n)^{\vee}$-valued  currents defined by
 			\begin{equation}
 			\alpha_{\bfy}(m, \bfv_1)(\varphi'') \ =   \ \varphi''_2(0) \  \lie g^o_{U_{\bfy}^{\perp}}(m, \bfv_1,\varphi''_1) \wedge \delta_{Z(\bfy) } \wedge \Omega^{n-\rk(T)}
 			\end{equation}
 			and
 			\begin{equation}
 				\beta_{\bfy}(m, \bfv_1)(\varphi'') \  = \  \varphi''_2(0) \  \omega_{U^{\perp}_{\bfy} }(m,\bfv_1,\varphi_1'') \wedge \delta_{Z(\bfy)} \wedge \Omega^{n-\rk(T)} - \delta_{Z_{\bfy}(m)(\varphi'') } \wedge \Omega^{n-\rk(T)}
 			\end{equation}
 			where $\varphi'' = \varphi''_1 \otimes \varphi''_2$ as before. 
 			
 			Turning to the class $\widehat A(T,\bfv)$, it can be readily verified  that
 			\begin{equation}
	 			\widehat A(T,\bfv) = \widehat{Z(T)}{}^{\can} \cdot \widehat{ \omega}^{n-\rk(T)} \ + \ \left( 0, [ \psi(T,\bfv) - \delta_{Z(T)} \wedge \Omega^{n-\rk(T)}, \lie a(T,\bfv)] \right)
 			\end{equation}
 			where the currents $\lie a(T,\bfv)$ and $\psi(T,\bfv)$ are defined in \eqref{eqn:lie a def} and  \eqref{eqn:psi current def}, respectively. 
 			Now, by the same argument as in \eqref{eqn:Z(T) cycle relation}, and under the assumption $\rank(T) = \rank(T_2)+1$, we have (as a $\Gamma$-invariant current on $\bbD$)
 			\begin{align}
 			\lie a(T,\bfv)(\varphi_1 \otimes \varphi_2) \ &= \ \sum_{\bfy \in \Omega(T_2)} \varphi_2(\bfy) \sum_{\substack{ \bfx_1 \in \Omega(T_1)  \\ \langle \bfx_1, \bfy \rangle = T_{12} }} \varphi_1(\bfx_1) \lie g^o(\sqrt{v_1} x_1 ) \wedge \delta_{\bbD^+_{y}} \wedge \Omega^{n-r(T)}  \\
 			&=  \sum_{\bfy \in \Omega(T_2)} \varphi_2(\bfy) \cdot \sum_{m \in F}	
 			\left( \sum_{\substack{ 
 					\bfx_1'' \in 	U_{\bfy}^{\perp} \\
 					\langle \bfx_1'', \bfx_1'' \rangle = m 
 			}} 
 			\varphi_1''(\bfx_1'') \,  \lie g^o(\sqrt{v_1} x_1'') \wedge \delta_{\bbD^+_{y}} \wedge \Omega^{n-r(T)} \right) \\
 			& \qquad \qquad  \qquad \qquad \times 	r\left( \psm{T_1 - m & T_{12} \\ T_{12'} & T_2} , \varphi_1' \right) ,
 			\end{align}
 			where we use the fact that $\lie g^o(\sqrt{v_1} x_1) \wedge \delta_{\bbD^+_y}$ only depends on the orthogonal projection $x_1''$ of $x_1$ onto $U_{y}^{\perp} = \sigma_1(U_{\bfy}^{\perp})$. Thus, as $S(L^n)^{\vee}$-valued currents on $X$, we obtain the identity
 			\begin{equation}
 				\lie a(T,\bfv)(\varphi_1 \otimes \varphi_2) = \sum_{\bfy \text{ mod } \Gamma}  \sum_{m \in F}  \alpha_{\bfy}(m,\bfv_1)(\varphi_1''\otimes \varphi_2'') \cdot \left\{ \varphi_2'(\bfy) \,  	r_{\bfy}\left( \psm{T_1 - m & T_{12} \\ T_{12'} & T_2} , \varphi_1'\right) \right\} 
 			\end{equation}
 			with $\varphi_i =\varphi_i' \otimes \varphi_i ''$ as above. 
 			
 			A similar argument gives 
 			\begin{equation}
 			\psi(T,\bfv)(\varphi) - \delta_{Z(T)(\varphi)} \wedge \Omega^{n-\rk(T)}  = \sum_{\bfy \text{ mod } \Gamma}    \sum_{m \in F} \beta_{\bfy}(m, \bfv_1)(\varphi_1'' \otimes \varphi_2'') \cdot \left\{\varphi_2'(\bfy) \, 	r_{\bfy} \left( \psm{T_1 - m & T_{12} \\ T_{12'} & T_2}, \varphi_1' \right) \right\},
 			\end{equation}
 			and so in total, we have
 			\begin{equation} \label{eqn:final AHat eqn}
	 			\widehat A(T,\bfv)(\varphi_1\otimes \varphi_2) = \sum_{\bfy \text{ mod } \Gamma}   \sum_m \widehat Z_{\bfy}(m, \bfv_1) (\varphi_1'' \otimes \varphi_2'')\cdot \widehat{\omega}^{n - \rk(T_2) - 1} \cdot 	 \left\{ 	\varphi_2'(\bfy) \, r_{\bfy}\left( \psm{T_1 - m & T_{12} \\ T_{12'} & T_2} \right) (\varphi_1')\right\}
 			\end{equation} 
 			whenever $\rank(T) = \rank(T_2) + 1$. 
 	
 			Now suppose $\rank(T) = \rank(T_2)$. Then for any tuple $(\bfx_1, \bfy) \in \Omega(T)$, we must have $\bfx_1 \in U_{\bfy}$, and in particular, the only terms contributing to the right hand side of \eqref{eqn:final AHat eqn} are those with $m=0$. On the other hand, we have
 			\begin{equation}
	 			\lie a(T,\bfv) = 0, \qquad \psi(T,\bfv) = \delta_{Z(T)} \wedge \Omega^{n-\rk(T)},
 			\end{equation}
 			and hence
 		\begin{equation}
	 		\widehat A(T,\bfv) = \widehat{Z(T)}{}^{\can} \cdot \widehat \omega^{n-\rk(T)};
 		\end{equation}
 			with these observations, it follows easily from unwinding definitions that \eqref{eqn:final AHat eqn} continues to hold in this case.
 			
 			Finally, the statement in the proposition follows by observing that the $T$'th $q$ coefficient on the right hand side of \eqref{eqn:PhiA Prop main eqn} is precisely the right hand side of \eqref{eqn:final AHat eqn}. 

		\end{proof}
	
	\end{proposition}
	
\begin{corollary} \label{thm:A series}
		The series $\widehat \phi_A(\bftau)$ is a Jacobi modular form of weight $\kappa := (p+2)/2$ and index $T_2$, in the sense of \Cref{definition of modularity}. 
		
	\begin{proof}
		Fix $\bfy \in \Omega(T_2)$. By \Cref{genus one modularity}, applied to the space $U_{\bfy}^{\perp}$, there exist finitely many $\widehat z_{\bfy,1}, \dots, \widehat z_{\bfy,r}\in \ChowHat{1}_{\bbC}(X_{\bfy})$, finitely many (elliptic) forms $f_{\bfy,1}, \dots, f_{\bfy, r} \in A_{\kappa}(\rho^{\vee}_{\Lambda^{\perp}_{\bfy}})$ and an element $g_{\bfy} \in A_{\kappa}(\rho^{\vee}_{\Lambda_{\bfy}^{\perp}}; D^*(X))$ such that the identity 
		\begin{equation}
			\sum_{m \in F} 		\widehat Z_{U_{\bfy}^{\perp}}(m, \bfv_1)   q^m = \sum_{i = 1}^r f_{\bfy, i}(\bftau_1) \widehat z_{\bfy,i}\ + \ a( g_{\bfy} (\bftau_1))
		\end{equation}
		holds at the level of $q$-coefficients; here $\bftau_1 \in \bbH_1^d$ and $\bfv_1  = \mathrm{Im}(\bftau_1)$. Moreover, from the proof of \Cref{genus one modularity}, we see that $g_{\bfy}(\tau)$ is smooth on $X$. 
		
		Therefore, applying \Cref{PhiA Prop} and unwinding definitions, we obtain the identity
		\begin{equation}
			\widehat \phi_{A}(\bftau) \ = \ \sum_{\substack{ \bfy \in \Omega(T_2) \\ \text{mod } {\Gamma} }} \  \sum_{i=1}^r  \left( F_{\bfy, i} (\bftau) \otimes \theta_{\bfy}(\bftau) \right) \widehat Z_{\bfy, i}  \ + \ a\left(  \left( G_{\bfy}(\bftau) \otimes \theta_{\bfy}(\bftau) \right) \wedge \delta_{Z(\bfy)} \wedge \Omega^{n-\rank(T_2) - 1} \right) 
		\end{equation}
		of formal generating series, where 
		\begin{equation}
			\widehat Z_{\bfy, i} \ := \ \pi_{\bfy, *} \left( \widehat z_{\bfy,i} \right)  \cdot \widehat \omega^{n - \rk(T_2) - 1} \ \in \ \ChowHat{n}_{\bbC}(X, \Dcur),
		\end{equation}
		and we promote the elliptic forms $f_{\bfy, i}$  and  $g_{\bfy}$ to $S((\Lambda_{\bfy}^{\perp})^n)^{\vee}$-valued functions by setting 
		\begin{equation}
			F_{\bfy, i}(\bftau)(\varphi) \ := \ \varphi_2(0) \cdot  f_{\bfy, i}(\bftau_1)(\varphi_1)   , \qquad  G_{\bfy}(\bftau) = \varphi_2(0)  \cdot g_{\bfy}(\bftau_1)(\varphi_1)   
		\end{equation}
		for $\varphi = \varphi_1 \otimes \varphi_2 \in S(\Lambda_{\bfy}^{\perp}) \otimes S((\Lambda_{\bfy}^{\perp})^{n-1}) $ and $\bftau = \psm{\bftau_1 & \bftau_{12} \\ \bftau_{12}' & \bftau_2}$. 
		
		It remains to show that  $F_{\bfy,i}(\bftau)\otimes \theta_{\bfy}(\bftau)$ and  $G_{\bfy}(\tau)\otimes \theta_{\bfy}(\bftau)$ are invariant under the slash operators \eqref{eqn:slash operators}  for elements of $\widetilde \Gamma^J$;  this can  be  verified directly using the generators  \eqref{eqn:Jacobi generator epsilon} -- \eqref{eqn:Jacobi generator S}, the modularity in genus one of  $f_{\bfy, i}$ and $g_{\bfy}$, and explicit formulas for the Weil representation (as in e.g. \cite{KudlaCastle}).
			
	\end{proof}
\end{corollary}

	\bibliographystyle{alpha}
	\bibliography{refs}

\end{document}